\documentclass[11pt]{amsart}
\usepackage[T1]{fontenc}
\usepackage{txfonts}
\usepackage{float}
\usepackage{amssymb,verbatim}
\usepackage[all]{xy}
\usepackage{subfig}
\usepackage{tikz}
\usepackage{color}
\usepackage{a4wide}
\usepackage{mathrsfs}
\usepackage{longtable}
\usepackage{hyperref}
\usepackage{wrapfig}
\usetikzlibrary{arrows}
\DeclareMathAlphabet{\mathpzc}{OT1}{pzc}{m}{it}

\newcommand{\R}{\mathbb{R}}
\newcommand{\C}{\mathbb{C}}
\newcommand\Z{\mathbb{Z}}
\newcommand{\N}{\mathbb{N}}
\newcommand{\Q}{\mathbb{Q}}

\newcommand{\Pb}{\mathbb{P}}

\newcommand{\T}{\mathrm{T}}

\newcommand{\fracs}{\mathfrak{s}}

\newcommand{\ff}{\mathbf{f}}

\newcommand{\xx}{\mathbf{x}}

\newcommand{\Ccal}{\mathcal{C}}
\newcommand{\Dcal}{\mathcal{D}}
\newcommand{\Ecal}{\mathcal{E}}

\newcommand{\Hcal}{\mathcal{H}}
\newcommand{\Ical}{\mathcal{I}}

\newcommand{\Kcal}{\mathcal{K}}
\newcommand{\Lcal}{\mathcal{L}}
\newcommand{\Mcal}{\mathcal{M}}

\newcommand{\Ocal}{\mathcal{O}}
\newcommand{\Pcal}{\mathcal{P}}

\newcommand{\Scal}{\mathcal{S}}

\newcommand{\Ucal}{\mathcal{U}}

\newcommand{\Xcal}{\mathcal{X}}

\newcommand{\Cc}{\mathcal{C}}
\newcommand{\Dc}{\mathcal{D}}
\newcommand{\Ec}{\mathcal{E}}
\newcommand{\Hc}{\mathcal{H}}
\newcommand{\Ic}{\mathcal{I}}

\newcommand{\Kc}{\mathcal{K}}
\newcommand{\Lc}{\mathcal{L}}

\newcommand{\Pc}{\mathcal{P}}

\newcommand{\Sc}{\mathcal{S}}

\newcommand{\Uc}{\mathcal{U}}

\newcommand{\NN}{\mathscr{N}}
\newcommand{\OO}{\mathscr{O}}

\newcommand{\CP}{\mathbb{P}_{\mathbb{C}}}
\newcommand{\SL}{{\rm SL}}

\newcommand{\Mod}{\mathcal{M}}

\newcommand{\vol}{{\rm vol}}

\newcommand{\omg}{\omega}

\newcommand{\Del}{\Delta}

\newcommand{\ol}{\overline}

\newcommand{\ra}{\rightarrow}

\newcommand{\strate}{\Omega^d\Mcal_{g,n}(\kappa)}

\newcommand{\stratesp}{\Omega^d\Mcal_{0,n}(\kappa)}
\newcommand{\pstratesp}{\Mcal_{0,n}(\kappa)}

\newcommand{\clpstrate}{\Pb\Omega^d\ol{\Mod}_{g,n}(\kappa)}
\newcommand{\clpstratesp}{\Pb\Omega^d\ol{\Mod}_{0,n}(\kappa)}
\newcommand{\blowupsp}{\widehat{\Mod}_{0,n}(\kappa)}

\newtheorem{Theorem}{Theorem}[section]
\newtheorem{Corollary}[Theorem]{Corollary}
\newtheorem{Lemma}[Theorem]{Lemma}
\newtheorem{Proposition}[Theorem]{Proposition}

\newtheorem{Definition}[Theorem]{Definition}

\theoremstyle{remark}
\theoremstyle{example}
\newtheorem{Remark}[Theorem]{Remark}
\begin{document}
\title[Incidence Variety for $d$-differentials in genus 0]{The Incidence Variety Compactification of strata of $d$-differentials in genus $0$}

\author[D.-M. Nguyen]{Duc-Manh Nguyen}


\address{Univ. Bordeaux, CNRS, Bordeaux INP, IMB, UMR 5251, F-33405 Talence, France}
\email[D.-M.~Nguyen]{duc-manh.nguyen@math.u-bordeaux.fr}


\date{\today}
\begin{abstract}
Given $d\in \mathbb{Z}_{\geq 2}$, for every $\kappa=(k_1,\dots,k_n) \in \mathbb{Z}^{n}$ such that $k_i\geq 1-d$ and $k_1+\dots+k_n=-2d$, denote by $\Omega^d\mathcal{M}_{0,n}(\kappa)$ and  $\mathbb{P}\Omega^d\mathcal{M}_{0,n}(\kappa)$ the corresponding stratum of $d$-differentials in genus $0$ and its projectivization respectively.
We specify an ideal sheaf of the structure sheaf of $\overline{\mathcal{M}}_{0,n}$ and show that the incidence variety compactification $\mathbb{P}\overline{\Omega}^d\mathcal{M}_{0,n}(\kappa)$ of $\mathbb{P}\Omega^d\mathcal{M}_{0,n}(\kappa)$ is isomorphic to the blow-up of $\overline{\mathcal{M}}_{0,n}$ along this sheaf of ideals. We also obtain an explicit divisor representative of the tautological line bundle on the incidence variety.
In an accompanying work~\cite{Nguyen23}, the construction of $\mathbb{P}\overline{\Omega}^d\mathcal{M}_{0,n}(\kappa)$ in this paper will be used to prove a recursive formula computing the volumes of the spaces of flat metric with fixed conical angles on the sphere.
\end{abstract}

\maketitle


\section{Introduction}\label{sec:intro}
\subsection{Statements of the main results}\label{sec:state:results}
Given a Riemann surface $X$, we  denote by $K_X$ its canonical line bundle. Let $g$ be a non-negative integer, and $d$ a positive one.
For any vector $\kappa=(k_1,\dots,k_n) \in \Z^n$ such that $k_i>-d$,  and $k_1+\dots+k_n=d(2g-2)$, denote by $\strate$ the space of tuples $(X,x_1,\dots,x_n,q)$, where $X$ is a compact Riemann surface of genus $g$, $\{x_1,\dots,x_n\}$ are $n$ marked points on $X$, and $q$ is a non-zero meromorphic $d$-differential on $X$ (that is a meromorphic section of the bundle $K^{\otimes d}_X$) such that
$$
\mathrm{div}(q)=\sum_{i=1}^n k_ix_i.
$$
The space $\strate$ is called a {\em stratum} of  $d$-differentials in genus $g$.
There is a natural action of $\C^*$ on $\strate$ which consists of multiplying the differential $q$ by a scalar in $\C^*$.
We denote by $\Pb\strate$ the quotient of $\strate$ by this $ \C^*$-action.

In this paper we investigate a natural compactification of strata of $d$-differentials in genus $0$, which is known as the {\em Incidence Variety Compactification} (see \cite{Gen18,BCGGM1,BCGGM2}).
Denote by $\ol{\Mod}_{0,n}$  the Deligne-Mumford-Knudsen compactification of $\Mod_{0,n}$. There is a holomorphic vector bundle $\ol{\Hc}^{(d)}_{0,n}$ over $\ol{\Mod}_{0,n}$ whose fiber over a point $(C,x_1,\dots,x_n) \in \ol{\Mod}_{0,n}$ is identified with the space $H^0(C,d\omg_C+\sum_{i=1}^n(d-1)x_i)$ ($\omg_C$ is the dualizing sheaf of $C$).
The restriction of $\ol{\Hc}^{(d)}_{0,n}$ to $\Mod_{0,n}$ is denoted by $\Hc^{(d)}_{0,n}$.
By definition, $\Pb\stratesp$ is a subset of  the projective bundle $\Pb\Hcal^{(d)}_{0,n}$ associated to $\Hcal^{(d)}_{0,n}$.
The Incidence Variety Compactification of $\Pb\stratesp$, which will be denoted by $\clpstratesp$,  is defined  to be its closure  in $\Pb\ol{\Hcal}^{(d)}_{0,n}$. 
Define
\begin{equation}\label{eq:def:weights}
\mu_i:=-\frac{k_i}{d}, \; i=1,\dots,n, \quad \text{ and } \quad \mu:=(\mu_1,\dots,\mu_n).
\end{equation}
We will call $\mu_i$ the weight of the $i$-th marked point on the pointed curves parametrized by $\ol{\Mod}_{0,n}$. Observe that $\mu_i < 1$, and since  $k_1+\dots+k_n=-2d$, we must have $\mu_1+\dots+\mu_n=2$. Throughout the paper, given $I \subset \{1,\dots,n\}$, we define
$$
\mu(I):=\sum_{i\in I}\mu_i.
$$

Recall that the boundary $\partial \ol{\Mod}_{0,n}:=\ol{\Mod}_{0,n}\setminus\Mod_{0,n}$ of $\ol{\Mod}_{0,n}$ is a simple normal crossing divisor.
The irreducible components of  $\partial \ol{\Mod}_{0,n}$ are in bijection with partitions of $\{1,\dots,n\}$ into two subsets $\{I_0,I_1\}$ such that $\min\{|I_0|,|I_1|\geq 2\}$.
The irreducible component of  $\partial \ol{\Mod}_{0,n}$ associated to $\Sc$ will be denoted by $D_\Sc$.
Denote by $\Pcal$ the set of partitions associated to the components of $\partial\ol{\Mod}_{0,n}$. For any $\Sc=\{I_0,I_1\}\in \Pcal$, by convention  we will always choose the numbering of the subsets $I_0,I_1$ such that
\begin{equation}\label{eq:conv:partition}
\mu(I_0) \leq 1 \leq \mu(I_1) \quad \Leftrightarrow \quad
\sum_{i\in I_0} k_i \geq -d \geq \sum_{i\in I_1} k_i.
\end{equation}
We will associate to each component $D_\Sc$ of $\partial\ol{\Mod}_{0,n}$ the weight $\mu_{\Sc}$ defined by
\begin{equation}\label{eq:def:weight:divisor}
\mu_\Sc:=\frac{1}{2}\cdot\left(\mu(I_1) - \mu(I_0))=1-\mu(I_0)\right).
\end{equation}
Note that  $\mu_\Sc$ is always non-negative.
We warn the reader that our convention on the numbering of $I_0,I_1$, and our definition of $\mu_\Sc$ are different from \cite{KN18}.

Consider now a point $\xx$ in $\ol{\Mod}_{0,n}$. Let $C_\xx$ denote the stable curve parametrized by $\xx$. Assume that  $C_\xx$ has $(r+1)$ irreducible components, which will be denoted by $C^0_\xx,\dots,C^r_\xx$.
Let $\NN(\xx)$ be the set of nodes of $C_\xx$. Note that $|\NN(\xx)|=r$.
For any $\alpha \in \NN(\xx)$, let $\Sc_\alpha:=\{I_{0,\alpha},I_{1,\alpha}\}$ be the associated partition of $\{1,\dots,n\}$.
Recall that $\dim\ol{\Mod}_{0,n}=n-3$.
Let $\Ucal_\xx$ be a neighborhood of $\xx$ which satisfies the following conditions
\begin{itemize}
 \item[(i)] $\Ucal_\xx$ does not intersect any boundary divisor $D_\Scal$ such that $\Scal \notin \{\Scal_\alpha, \, \alpha \in \NN(\xx)\}$,

 \item[(ii)] $\Ucal_\xx$ can be identified with an open subset of $\C^{n-3}$ such that, for each $\alpha \in \NN(\xx)$, there is a coordinate function $t_\alpha$ such that $D_{\Scal_\alpha}\cap \Ucal_\xx$ is defined by the equation $t_\alpha=0$.
\end{itemize}
As $\ol{\Mod}_{0,n}$ is a projective variety, we can actually choose $\Ucal_\xx$ to be an open affine of $\ol{\Mod}_{0,n}$ such that $t_\alpha$ are elements of the coordinate ring of $\Ucal_\xx$.

Splitting a node $\alpha \in \NN(\xx)$ into two points, we obtain two subcurves $\hat{C}^{0}_{\xx,\alpha}, \hat{C}^{1}_{\xx,\alpha}$ of $C_\xx$, where $\hat{C}^{k}_{\xx,\alpha}$ contains  the $i$-th marked points with $i\in I_{k,\alpha}$.
For each $j\in\{0,\dots,r\}$, define
\begin{equation}\label{eq:def:beta}
\beta_{j,\alpha}=\left\{ \begin{array}{ll}
                          d\mu_{\Scal_\alpha} & \text{ if } C^j_\xx \subset \hat{C}^{0}_{\xx,\alpha}, \\
                          0 & \text{ otherwise}.
                         \end{array} \right.
\end{equation}
Let
\begin{equation}\label{eq:def:t:beta}
\beta_j=(\beta_{j,\alpha})_{\alpha \in \NN(\xx)} \in (\Z_{\geq 0})^{\NN(\xx)}, \quad \text{ and } \quad t^{\beta_j}:=\prod_{\alpha \in \NN(\xx)}t_\alpha^{\beta_{j,\alpha}}.
\end{equation}
Let  $\Ocal_{\ol{\Mod}_{0,n}}$ be the structure sheaf of $\ol{\Mod}_{0,n}$, and $\Ical_{\Ucal_\xx}$ be the ideal sheaf of $\Ocal_{\ol{\Mod}_{0,n}|\Ucal_\xx}$ generated by $\{t^{\beta_0},\dots,t^{\beta_r}\}$. We will prove
\begin{Theorem}\label{th:IVC:is:blowup}\hfill
\begin{itemize}
\item[(i)] The family $\{\Ical_{\Ucal_\xx}, \, \xx \in \ol{\Mod}_{0,n}\}$ defines a sheaf of ideals $\Ical$ of $\Ocal_{\ol{\Mod}_{0,n}}$.

\item[(ii)] The incidence variety compactification $\clpstratesp$ of $\Pb\stratesp$ is isomorphic to the blow-up $\blowupsp$ of $\ol{\Mod}_{0,n}$ along $\Ical$.


\item[(iii)] Every irreducible component of $\partial\blowupsp:=\blowupsp\setminus\Mod_{0,n}$ is a divisor. The set of irreducible components of $\partial\blowupsp$ is in bijection with the set $\hat{\Pcal}(\mu)$ of partitions $\Scal$ of $\{1,\dots,n\}$ such that up to a renumbering of the subsets in $\Scal$ either
\begin{itemize}
\item[$\bullet$] $\Scal=\{I_0,I_1\} \in \Pc$,  or

\item[$\bullet$] $\Scal=\{I_0,I_1,\dots,I_r\}$, with $r\geq 2$, $\mu(I_0) <1$, and $\mu(I_j)>1$ for all $j=1,\dots,r$.
\end{itemize}
\end{itemize}
\end{Theorem}

The irreducible component of $\partial\blowupsp$ associated to $\Sc\in \hat{\Pc}(\mu)$ will be denoted by $\hat{D}_\Sc$.
By construction $\hat{p}^{-1}\Ical\cdot\Ocal_{\blowupsp}$, where $\Ic$ is the ideal sheaf in Theorem~\ref{th:IVC:is:blowup}~(i), gives a Cartier divisor in  $\blowupsp$, which will be called the {\em exceptional divisor} and denoted by $\Ecal$.
Our second main result gives  explicit expressions of $\Ec$
and the restriction of the tautological line bundle on $\Pb\ol{\Hc}^{(d)}$ to $\clpstratesp\simeq \blowupsp$.
\begin{Theorem}\label{th:vol:n:inters:g0}\hfill
\begin{itemize}
\item[(a)] For all $\Sc=\{I_0,I_1,\dots,I_r\} \in \hat{\Pc}(\mu)$, let us   define $m(\Sc):=d^r\cdot \prod_{j=1}^r(\mu(I_j)-1)$.
Then the  Weil divisor associated to $\Ec$ is given by
\begin{equation}\label{eq:except:Weil:div}
\Ec \sim \sum_{\Sc \in \hat{\Pc}(\mu)}(|\Sc|-2)\cdot m(\Sc)\cdot\hat{D}_\Sc
\end{equation}
where $|\Sc|$ is the length of $\Sc$.

\item[(b)] Consider the following $\Q$-divisor in $\ol{\Mod}_{0,n}$
$$
\Dc_\mu := \frac{d}{(n-2)(n-1)}\sum_{\Scal=\{I_0,I_1\}\in \Pcal}(|I_0|-1)(|I_1|-1-(n-1)\mu_\Scal)\cdot D_\Scal.
$$
Let $\hat{\Lc}_\mu$ denote the restriction of the tautological line bundle $\OO(-1)_{\Pb\ol{\Hcal}^{(d)}_{0,n}}$ to $\blowupsp$. Then we have
\begin{equation}\label{eq:tauto:ln:bdl:expr}
\hat{\Lc}\sim \hat{p}^*\Dc_\mu+\Ecal.
\end{equation}

\item[(c)] Let $\hat{\Dc}_\mu$ denote the $\Q$-divisor $\hat{p}^*\Dc_\mu+\Ecal$ in $\blowupsp$. Assume that $d$ does not divide $k_i$, for all $i=1,\dots,n$, then
\begin{equation}\label{eq:vol:n:inters:g0}
\vol_1(\Pb\stratesp)=\frac{(-1)^{n-3}}{d^{n-3}}\cdot\frac{(2\pi)^{n-2}}{2^{n-2}(n-2)!}\cdot\hat{\Dc}_\mu^{n-3}
\end{equation}
where $\vol_1$ is the volume form on $\Pb\stratesp$ defined as in \cite[\textsection 6]{Ng22}, and  $\hat{\Dc}_\mu^{n-3}$ means the self-intersection number $\underbrace{\hat{\Dc}_\mu\cdots\hat{\Dc}_\mu}_{n-3}$ of $\hat{\Dc}_\mu$ in $\blowupsp$.
\end{itemize}
\end{Theorem}
\begin{Remark}\label{rk:th:vol}
In the  cases where $k_i <0$, for all $i=1,\dots,n$, we have $\blowupsp\simeq\ol{\Mod}_{0,n}$, $\hat{\Dc}_\mu\simeq \Dc_\mu$, and \eqref{eq:vol:n:inters:g0} is the content of \cite[Th. 1.1]{KN18}. For general $\kappa$, the fact that $\hat{\Dc}^{n-3}_\mu$ computes the volume of $\Pb\stratesp$ follows from the results of \cite{CMZ19, Ng22}. The novelty of Theorem~\ref{th:vol:n:inters:g0} is that we have an explicit expression of the divisor $\hat{\Dc}_\mu$.
\end{Remark}

\subsection{Motivations and  related works}\label{sec:motivations}
Abelian differentials and quadratic differentials on Riemann surfaces are central objects  in  Teichm\"uller theory, dynamics in moduli spaces, billiards in rational polygons, interval exchange transformations among others. For some surveys on these fields of research, we refer to \cite{MT02, Z:survey, W:survey}.
From the flat metric  point of view, $d$-differentials are natural generalizations of translation surfaces and half-translation surfaces, the objects associated to Abelian and quadratic differentials.
For $d\in\{1,2,3,4,6\}$,  $d$-differentials also arise from the counting problem of triangulations and quadrangulations on surfaces (see for instance \cite{Th98, ES18, Engel-I, KN20}).


In view of applications, it is important to have an adequate compactification of $\Pb\strate$.
The incidence variety compactification $\clpstrate$ of strata of Abelian differentials and strata of $d$-differentials were introduced  in \cite{Gen18, BCGGM1,BCGGM2} (see also~\cite{Sau20}).
The characterizing properties of points in $\clpstrate$ are given in \cite{BCGGM1} and \cite{BCGGM2}.
However, the local structure and the global description of $\clpstrate$ as algebraic varieties are largely  unknown.
To the author's knowledge, this is the first time an explicit global description of some incidence varieties is given.

In the accompanying paper \cite{Nguyen23} a recursive formula relating the volumes of strata $\pstratesp$ where none of the $k_i$ is divisible by $d$ is deduced from the results of this paper.
As an application, we partially recover a formula for the Masur-Veech volumes of strata of quadratic differentials in  genus $0$ which was conjectured by Kontsevich and proved by Athreya-Eskin-Zorich \cite{AEZ16}.
We are also hopeful that the analysis in this paper could be generalized to the investigation of incidence varieties in higher genus.


Recently, other  compactifications of strata of Abelian differentials and $d$-differentials in general that are complex orbifolds are constructed in  \cite{BCGGM:multiscale} and \cite{CMZ19}.
The objects parametrized by these compactifications are called {\em multi-scale differentials}.
By its very definition, a multi-scale differential  comes equipped with a level structure on its dual graph.
Such  level structures are not involved in the construction of the incidence variety compactification.
However, on every stable curve underlying a $d$-differential in the incidence variety, there exists implicitly some order relation between the irreducible components, and the components on which the $d$-differentials does not vanish identically correspond to the maximal elements of this relation. In other words, while a multi-scale differential compactification records information on all levels of a compatible level structure, the incidence variety compactification only records information on the top level components.
For a more detailed account on the relation between these two compactifications, we refer to \cite{CGHMS22}.


\subsection{Organization} The paper is organized as follows: in \textsection\ref{sec:embed:Mod:0:n} we introduce the bundles  $\ol{\Hc}^{(d)}_{0,n}$ and $\Pb\ol{\Hc}^{(d)}_{0,n}$ over $\ol{\Mod}_{0,n}$. The space $\Pb\stratesp$ can be seen as the image of a section $\ff$ of $\Pb\ol{\Hc}_{0,n}$ defined over $\Mod_{0,n}$. Using a specific local construction of the universal curve $\ol{\Cc}_{0,n}$, we  provide some candidates for the extension of $\ff$ to $\ol{\Mod}_{0,n}$.

In \textsection\ref{sec:blowup:M0n}, we define the ideal sheaf $\Ic$ and construct $\blowupsp$.
Using the results in \textsection\ref{sec:blowup:M0n}, in \textsection\ref{sec:bdry:blowup} we show that every irreducible component of $\partial\blowupsp$ corresponds uniquely to a partition in $\hat{\Pc}(\mu)$.
We investigate the geometric interpretation of the elements of $\blowupsp$ in \textsection\ref{sec:geom:interpret}. This allows us to define an embedding $\hat{\ff}: \blowupsp \to \Pb\ol{\Hcal}^{(d)}_{0,n}$ extending the section $\ff$ in \textsection\ref{sec:embed:Mod:0:n}, showing that $\blowupsp$ is isomorphic to the closure of $\Pb\stratesp$ in $\Pb\ol{\Hcal}^{(d)}_{0,n}$.

In \textsection\ref{sec:except:div}, we compute the vanishing order of the Cartier divisor $\Ec$ along the components of $\partial\blowupsp$ proving \eqref{eq:except:Weil:div}.
To prove \eqref{eq:tauto:ln:bdl:expr}, in \textsection\ref{sec:kaw:ln:bdl}
we investigate the Kawamata line bundle $\Kc_\mu$ over $\ol{\Cc}_{0,n}$. We   show that $\Kcal_\mu$ is the pullback of a line bundle $\bar{\Lcal}_\mu$ on $\ol{\Mod}_{0,n}$, whose divisor representative is equivalent to $D_\mu$.
The proof of Theorem~\ref{th:vol:n:inters:g0} is completed in  \textsection\ref{sec:proof:vol:n:inters:nb}.

\subsection*{Acknowledgements}
The author is grateful to Yohan Brunebarbe and Vincent Koziarz  for the enlightening and inspiring discussions, which played an important role in the realization of this work.
He thanks  Martin M\"oller and Adrien Sauvaget for the stimulating discussions.
The author is partly supported by the French ANR project ANR-19-CE40-0003.

\section{Embedding $\Mod_{0,n}$ into a projective bundle} \label{sec:embed:Mod:0:n}
\subsection{The space of finite area $d$-differentials}\label{subsec:d:diff:bundle}
Let $p: \ol{\Cc}_{0,n}\to \ol{\Mod}_{0,n}$ be the universal curve over $\ol{\Mod}_{0,n}$, and $K_{\ol{\Cc}_{0,n}/\ol{\Mod}_{0,n}}$ the relative canonical line bundle. Given $\xx \in \ol{\Mod}_{0,n}$, the fiber of $p$ over $\xx$ is denoted by $C_\xx$.
Let $\Gamma_i, \, i=1,\dots,n$, be the image of $\ol{\Mod}_{0,n}$ by the $i$-th tautological section of $p$. We abusively denote by $\Gamma_i$ the associated line bundle over $\ol{\Cc}_{0,n}$.
Consider the line bundle
$$
\Kcal^{(d)}_{0,n}:=d\cdot K_{\ol{\Ccal}_{0,n}/\ol{\Mod}_{0,n}}+\sum_{i=1}^n(d-1)\cdot\Gamma_i
$$
over $\ol{\Cc}_{0,n}$.
The restriction of $\Kc^{(d)}_{0,n}$ to $C_\xx$, denoted by $\Kc^{(d)}_{\xx}$, satisfies
\[
\Kcal^{(d)}_{\xx} \sim d\cdot\omega_{C_{\xx}}+\sum_{i=1}^n(d-1)\cdot\{x_i\}
\]
where $\omega_{C_{\xx}}$ is the dualizing sheaf of $C_{\xx}$, and $x_i=C_\xx\cap \Gamma_i, \; i=1,\dots,n$.

As usual, let $C^0_\xx,\dots,C^r_\xx$ be the irreducible components of $C_\xx$. Let $\varphi: \tilde{C}_{\xx} \ra C_{\xx}$ be the normalization map of $C_{\xx}$. For $k=0,\dots,r$, the connected component of $\tilde{C}_{\xx}$  corresponding to $C^k_{\xx}$ is denoted by $\tilde{C}^k_{\xx}$.
We abusively denote by $\{x_1,\dots,x_n\}$ the preimages in $\tilde{C}_{\xx}$  of the marked points in $C_{\xx}$.
For each $k\in \{0,\dots,r\}$, denote by $P_k$ the set $\tilde{C}^k_{\xx}\cap \{x_1,\dots,x_n\}$, and by $N_k$ the set of points in $\tilde{C}^k_{\xx}$ that are nodes of $C_{\xx}$.
One can readily check that
\[
\varphi^*(\omega_{C_{\xx}}-\Kcal^{(d)}_{\xx})_{|\tilde{C}^k_{\xx}} \sim  (1-d) \left(\omega_{\tilde{C}^k_{\xx}}+\sum_{y_j \in N_k} \{y_j\} + \sum_{x_i\in P_k}\{x_i\}\right).
\]
In particular
\[
\mathrm{deg}\varphi^*(\omega_{C_{\xx}}-\Kcal^{(d)}_{\xx})_{|\tilde{C}^k_{\xx}}=(1-d)(|N_k|+|P_k|-2) <0,
\]
which implies that $\dim H^0(\tilde{C}_{\xx},\varphi^*(\omega_{C_{\xx}}-\Kcal^{(d)}_{\xx}))=0$. Therefore
\begin{equation}\label{eq:R1:vanish}
 \dim H^1(C_{\xx},\Kcal^{(d)}_{\xx})=\dim H^0(C_{\xx},\omega_{C_{\xx}}-\Kcal^{(d)}_{\xx})=0
 \end{equation}
(see for instance \cite[Chap.10, \textsection 2]{ACG2011}).
By Riemann-Roch theorem, \eqref{eq:R1:vanish} implies that
\begin{equation}\label{eq:dim:Kd:x}
\dim H^0(C_{\xx},\Kcal^{(d)}_{\xx})=\deg(\Kcal^{(d)}_{\xx})+1=(d-1)(n-2)-1.
\end{equation}
In particular, we see that $\dim H^0(C_{\xx},\Kcal^{(d)}_{\xx})$ does not depend on $\xx$. Hence, by a classical result (see for instance \cite[Chap.III, Cor.12.9]{Hart}), $p_*\Kcal^{(d)}_{0,n}$ is a vector bundle $\ol{\Hcal}^{(d)}_{0,n}$ of rank $(d-1)(n-2)-1$ over $\ol{\Mod}_{0,n}$.
The restriction of $\ol{\Hcal}^{(d)}_{0,n}$ to $\Mod_{0,n}$ will be denoted by $\Hcal^{(d)}_{0,n}$.
Denote by  $\Pb\ol{\Hcal}^{(d)}_{0,n}$  and $\Pb\Hcal^{(d)}_{0,n}$ the  projective bundle associated with $\ol{\Hcal}^{(d)}_{0,n}$ and $\Hcal^{(d)}_{0,n}$ respectively.

It is a well known fact that every meromorphic $d$-differentials with poles of order at most $d-1$ defines a flat metric with conical singularities at the zeros and poles on the underlying Riemann surface, whose total area is finite. For this reason, we will call elements of $\Hc^{(d)}_{0,n}$ (resp. of $\Pb\Hc^{(d)}_{0,n}$)   {\em finite area}  (resp. {\em projectivized finite area}) $d$-differentials.

Every $d$-differential $q$ in the stratum $\stratesp$ is a meromorphic section of the line bundle $d\cdot\omega_{C_\xx}$ over the curve $C_\xx$ for some $\xx\in\Mod_{0,n}$. Since all the poles and zeros of this section are located at the marked points of $C_\xx$, and $k_i \geq 1-d$ for all $i=1,\dots,n$, the $d$-differential $q$ can be seen as an element of $H^0(C_\xx, d\cdot\omega_{C_\xx}+\sum_{i=1}^n(d-1)\cdot x_i) \simeq H^0(C_\xx, \Kc^{(d)}_\xx)$.  This means that $\stratesp$ is contained in the total space of the bundle $\Hc^{(d)}_{0,n}$, and therefore $\Pb\stratesp$ is a subvariety of $\Pb\Hc^{(d)}_{0,n}$.

For every $\xx \sim (\CP^1,x_1,\dots,x_n) \in \Mod_{0,n}$, the $d$-differential  $q_\xx:=\prod_{i=1}^n(z-x_i)^{k_i}(dz)^d$, where $z$ is the inhomogeneous coordinate on $\CP^1$, is an element of $\stratesp$. This $d$-differential is not well defined since it depends on the identification $C_\xx \sim (\CP^1,x_1,\dots,x_n)$. However, the line generated $q_\xx$ in $H^0(C_\xx,\Kc^{(d)}_\xx)$, which will be denoted by $[q_\xx]$, only depends on $\xx$.
We thus get a section of the bundle $\Pb\Hc^{0,n} \to \Mod_{0,n}$ given by
\[
\begin{array}{cccc}
\ff:& \Mod_{0,n} & \to & \Pb\Hc^{(d)}_{0,n}\\
    &  \xx       & \mapsto & [q_\xx]
\end{array}
\]
By definition, the Incidence Variety Compactification $\clpstratesp$ of $\Pb\stratesp$ is  the closure of $\Pb\stratesp$ in $\Pb\ol{\Hc}^{(d)}_{0,n}$.

\subsection{A local construction of the universal curve}\label{subsec:univ:curv:near:bdry}
To our purpose, we start by constructing the universal curve over a neighborhood of any point $\xx \in \ol{\Mod}_{0,n}$.
Assume that $\xx$ belongs to a stratum of codimension  $r \geq 1$  in $\partial \ol{\Mcal}_{0,n}$.
This means that $\xx$ is contained in the intersection of $r$ boundary divisors.
Let $x_1,\dots,x_n$ be the marked points on $C_\xx$.
The stable curve $(C_\xx,x_1,\dots,x_n)$ parametrized by $\xx$ has $(r+1)$ components denoted by $C^0_\xx,\dots,C^{r}_\xx$.
Let $I_j \subset \{1,\dots,n\}$ be the set of indices of the marked points contained in $C_\xx^j$. 
The topology of $C_\xx$ is encoded by its dual graph $\T_\xx$, which is a tree since $C_\xx$ has genus zero.
Each component $C^j_\xx$ of $C_\xx$ corresponds to a vertex of $\T_\xx$ which is denoted by $v_j$.
Each node of $C_\xx$  corresponds to an edge of $\T_\xx$. Since $\T_\xx$ is a tree, each of its edges is uniquely determined by an unordered pair $\{j,j'\}$, with $j,j' \in \{0,\dots,r\}$ and $j\neq j'$.
Let $\NN(\xx)$ denote the set pairs $\{j,j'\}$ associated with the nodes of $C_\xx$.
The point on $C^j_\xx$  that corresponds to a node $\{j,j'\}$ will be denoted by $y_{jj'}$.
For all $j\in \{0,\dots,r\}$, let
$$
\NN_j(\xx):=\{j'\in \{0,\dots,r\}, \; \{j,j'\} \in \NN(\xx)\}, \quad \text{ and } \quad n_j:=|I_j|+|\NN_j(\xx)|.
$$
Let $\xx_j \in \Mcal_{0,n_j}$ be the pointed curve  $(C^j_\xx,(x_i)_{i\in I_j},(y_{jj'})_{j'\in \NN_j(\xx)})$.
A neighborhood of $\xx$ in its stratum is identified with $U_0\times\dots\times U_{r}$, where $U_j$ is a neighborhood of
$\xx_j$ in $\Mcal_{0,n_j}$.
We can identify  $U_j$ with a neighborhood of $0\in \C^{n_j-3}$.

Let $u_j=(u_{j,1},\dots,u_{j,n_j-3})$ be a system  of holomorphic coordinates on $U_j$.
For all $u_j \in U_j$, the pointed curve parametrized by $u_j$ is isomorphic to $(\CP^1,(a_{ji})_{i\in I_j},(b_{jj'})_{j'\in \NN_j})$, where $a_{ji}$ is the marked point corresponding to $x_i$ and $b_{jj'}$ is the point corresponding to the node $y_{jj'}$. We consider $a_{ji}$ and $b_{jj'}$ as holomorphic functions of $u_j$.
Note that the points $\{a_{ji}(u_j), \, i\in I_j\}$ and  $\{b_{jj'}(u_j), \, j'\in \NN_j(\xx)\}$ are pairwise distinct.

Let $\Delta$ be a small disc about $0\in \C$.
The set $ \Ucal:=\Delta^{\NN(\xx)}\times U_0\times \dots \times U_r$
can be identified with an open neighborhood of $\xx$ in $\ol{\Mod}_{0,n}$.
Our goal now is to construct a family of $n$-pointed genus zero curves over $\Ucal$ degenerating to $C_\xx$.
Consider
$$
\Xcal:=\underbrace{\CP^1\times\dots\times\CP^1}_{r+1}\times \Ucal. 
$$
Let $z_{j-1}$ be the inhomogeneous coordinate on the $j$-th $\CP^1$ factor of $\Xcal$.
For any $\alpha=\{j,j'\} \in \NN(\xx)$, the coordinate on the $\alpha$-factor of $\Delta^{\NN(\xx)}$ will be denoted by $t_\alpha$.
For $j=0,\dots,r$, let $\pi_j: \Xcal \ra \CP^1$ be the projection from $\Xcal$ onto its $(j+1)$-th $\CP^1$ factor.
Let us write $z=(z_0,\dots,z_r)$, $t=(t_\alpha)_{\alpha\in \NN(\xx)}$, and $u=(u_0,\dots,u_r)$. Define
\begin{equation}\label{eq:mult:nodes:fam:curv}
\ol{\Ccal}_{\Ucal}:= \{(z,t,u) \in \Xcal, \; (z_j-b_{jj'})(z_{j'}-b_{j'j})=t_{\alpha}, \; \forall \alpha=\{j,j'\} \in \NN(\xx)\}.
\end{equation}
Let $p: \ol{\Ccal}_{\Ucal} \ra \Ucal$ be the  natural projection.
Given $(t,u)\in \Delta^{\NN(\xx)}\times U_0\times\dots\times U_r$,
let $C(t,u)$ be the curve in $(\CP^1)^{r+1}$ defined by the equations \eqref{eq:mult:nodes:fam:curv}. By construction, we have $C(0,0)\simeq C_\xx$.

Let $Z(t):=\{\alpha \in \NN(\xx), \; t_\alpha=0\}$.
The elements of $Z(t)$ are in bijection with the  nodes of $C(t,u)$.
Remove from the tree $\T_\xx$ the edges representing the elements of $Z(t)$, we get  $(r(t)+1)$ subtrees of $\T_\xx$, where $r(t):=|Z(t)|$.
Denote those subtrees by $\T^k_\xx(t), k=0,\dots,r(t)$.
Let $J_0(t),\dots,J_r(t)$ denote the partition of $\{0,1,\dots,r\}$ such that $j\in J_k(t)$ if and only if $v_j\in \T_\xx^k(t)$, where $v_0,\dots,v_r$ are the vertices of $\T_\xx$.
The following lemma provides some basic properties  of $C(t,u)$, its proof is left to the reader.
\begin{Lemma}\label{lm:univ:curv:geom:codim2}\hfill
 \begin{itemize}
 \item[(i)]  The curve $C(t,u)$ has $r(t)+1$ irreducible components, each of which is associated with a tree in the family $\{\T^0_\xx(t),\dots,\T^{r(t)}_\xx(t)\}$.
 The component associated with $\T^k_\xx(t)$ will be denoted by $C^k(t,u)$.

  \item[(ii)] Given $k \in \{0,\dots,r(t)\}$, for every $j\in J_k(t)$, the restriction of $\pi_{j}$ to $C^{k}(t,u)$ realizes an isomorphism from $C^k(t,u)$ onto $\CP^1$.

  \item[(iii)] If $j\notin J_k(t)$, then the projection $\pi_{j}$ maps $C^{k}(t,u)$ to a point that corresponds to a node of $C(t,u)$.
 \end{itemize}
\end{Lemma}
We now construct  the tautological sections  of $p$.
For every  $i \in \{1,\dots,n\}$, there is a unique $j \in \{0,\dots,r\}$ such that $i\in I_j$.
Recall that $a_{ji}: U_{j} \ra \CP^1$ is the holomorphic fuction associated to the marked point $x_i \in C^j_\xx$.
Let $k \in \{0,\dots,r(t)\}$ be such that $j\in J_k(t)$.
By Lemma~\ref{lm:univ:curv:geom:codim2}(ii) the projection $\pi_j$ restricts to an isomorphism from $C^k(t,u)$ onto $\CP^1$.
Denote by $\varphi_{ji}(t,u)$ the unique point on $C^k(t,u)$ that is mapped to $a_{ji}(u_j)$ by $\pi_j$.
Define
\begin{equation}\label{eq:tauto:sect:U:def}
 \begin{array}{cccc}
  \sigma_i: & \Ucal & \ra  & \ol{\Ccal}_{\Ucal}\\
                   &(t,u) & \mapsto & (\varphi_{ji}(t,u),t,u).
 \end{array}
\end{equation}
Then the pointed curve parametrized by $(t,u) \in \Ucal$ is isomorphic to $(C(t,u),(\sigma_i(t,u))_{1\leq i \leq n})$.
For all $i\in \{1,\dots,n\}, \ell \in\{0,\dots,r\}$, $\pi_{\ell}\circ\sigma_i(t,u)$ is the $(\ell+1)$-th coordinate of the $i$-th marked point on $C(t,u)\subset \left(\CP^1\right)^{r+1}$.
By definition, if $i \in I_j$, then $a_{ji}=\pi_j\circ \sigma_i(t,u)$. In what follows, for all  $\ell \in \{0,\dots,r\}$, we will write
$$
a_{\ell i}(t,u):=\pi_\ell\circ \sigma_i(t,u).
$$
We summarize the construction above by the following
\begin{Proposition}\label{prop:univ:curv:model}
  The set $\Uc=\Del^{\NN(\xx)}\times U_0\times\dots\times U_r$ is isomorphic to a neighborhood of $\xx$ in $\ol{\Mod}_{0,n}$, over which the universal curve is identified with the projection $p: \ol{\Cc}_{\Uc} \to \Uc$.
  The section of $p$ that corresponds to the $i$-th marked point is given by the map $\sigma_i$ defined in \eqref{eq:tauto:sect:U:def}.
\end{Proposition}

\subsection{Sections of the twisted relative pluricanonical line bundle over the universal curve}\label{subsec:ratio:sect:rel:can}
In what follows $\xx$ will be a point in a stratum of codimension $r$ in $\ol{\Mod}_{0,n}$. We keep using the notations of \textsection\ref{subsec:univ:curv:near:bdry}.
Recal that
\begin{equation*}
t^{\beta_j}=\prod_{\alpha \in \NN(\xx)}t_\alpha^{\beta_{j,\alpha}}, \quad \text{ where }
\beta_{j,\alpha}=\left\{ \begin{array}{ll}
                          d\mu_{\Scal_\alpha} & \text{ if } C^j_\xx \subset \hat{C}^{0}_{\xx,\alpha}, \\
                          0 & \text{ otherwise},
                         \end{array} \right.
\end{equation*}
($\hat{C}^0_{\xx,\alpha}$ is the subcurve of $C_\xx$ that contains the marked points $x_i$ with $i \in I_{0,\alpha}$). Our goal is to show

\begin{Proposition}\label{prop:Kaw:trivial:ratio}
Let $\Uc^*:=\Uc\cap \Mod_{0,n}$, that is the subset of $\Uc$ which parametrizes smooth curves of genus $0$ with $n$ marked points.
For $j=0,\dots,r$, define
\begin{equation}\label{eq:def:Phi}
\Phi_j(t,u)=\prod_{i=1}^n (z_j-a_{ji})^{k_i}(dz_j)^d.
\end{equation}
Then $t^{\beta_j}\Phi_j$ is a {\em holomorphic } section of the bundle $\Kc^{(d)}_{0,n}$ over $\ol{\Cc}_\Uc \simeq p^{-1}(\Uc) \subset \ol{\Cc}_{0,n}$, whose restriction to every fiber $C_{\xx'}$ of $p$ with $\xx'\in \Uc^*$ is a $d$-differential in $\stratesp$.
Moreover, for all $j, j'\in \{0,1,\dots,r\}$, there is a holomorphic non-vanishing function $f_{jj'}$ on $\Ucal$ such that
$$
\frac{t^{\beta_j}\Phi_j}{t^{\beta_{j'}}\Phi_{j'}}=f_{jj'}.
$$
\end{Proposition}

We start by

\begin{Lemma}\label{lm:ratio:sect:Kaw:ln:bdl:2:comp}
For all $j,k\in \{0,\dots,r\}$, there is a  non-vanishing holomorphic function $f_{jk}$ on $\Ucal$ such that
$$
\frac{\Phi_j}{\Phi_{k}}=\frac{t^{\beta_k}}{t^{\beta_j}}f_{jk}
$$
where $\Phi_j$ and $\Phi_k$ are considered as meromorphic sections of the relative pluricanonical line bundle $K^{\otimes d}_{\ol{\Ccal}_{0,n}/\ol{\Mcal}_{0,n}}$.
\end{Lemma}
\begin{proof}
Since $C_\xx$ is connected it is enough to show the lemma in the case there is a node $\alpha$ between the  components $C^j_\xx$ and $C^{k}_\xx$.
Without loss of generality, we can assume that $C^j_\xx \subset \hat{C}^1_{\xx,\alpha}$ and $C^k_\xx \subset \hat{C}^0_{\xx,\alpha}$.
In what follows, to lighten the notations, the subscript $\alpha$ will be omitted.
The relations
$$
z_j=b_{jk}+\frac{t}{z_k-b_{kj}} \quad  \text{ and }  \quad a_{ji}=b_{jk}+\frac{t}{a_{ki}-b_{kj}}
$$
imply
\begin{equation}\label{eq:rel:coord:through:node}
z_j-a_{ji}=-t\frac{z_k-a_{ki}}{(z_k-b_{kj})(a_{ki}-b_{kj})}.
\end{equation}
Since $(z_j-b_{jk})(z_k-b_{kj})=t$ (see \eqref{eq:mult:nodes:fam:curv}),  we have $dz_j=\frac{-t}{(z_k-b_{kj})^2}dz_k$ as sections of $K_{\ol{\Ccal}_{\Ucal}/\Ucal}$. Hence
\begin{eqnarray*}
\Phi_j & = & \frac{(-t)^{\sum_{i=1}^n k_i} \prod_{i=1}^n(z_k-a_{ki})^{k_i}}{ (z_k-b_{kj})^{\sum_{i=1}^n k_i}\prod_{i=1}^n(a_{ki}-b_{kj})^{k_i}} \cdot\frac{(-t)^d}{(z_k-b_{kj})^{2d}}(dz_k)^d\\
       & = & \frac{(-1)^d t^{-d}}{\prod_{i=1}^n(a_{ki}-b_{kj})^{k_i}}\Phi_k
\end{eqnarray*}
(here we used the equality $\sum_{i=1}^n k_i=-2d$). Recall that $a_{ki}$ is the $(k+1)$-th coordinate of the marked point $\sigma_i(t,u) \in C(t,u)$. We have two possibilities
\begin{itemize}
\item $i \in I_{0}$, that is $x_i \in \hat{C}^0_{\xx}$:  if $x_i \in C^k_\xx$, then since the marked point $x_i$ and the point corresponding to the node $\alpha$ on $C^k_\xx$ are distinct, $(a_{ki}-b_{kj})$ is a holomorphic non-vanishing function on $U_k$ (hence on $\Uc$). Otherwise, $x_i$ is contained in another component $C^{k'}_\xx$ of $\hat{C}^0_{\xx}$. From Lemma~\ref{lm:univ:curv:geom:codim2}(iii), we know that the projection $\pi_k$ maps $\sigma_i(t,u)$ to a point close to a node $b_{kj'}$ which must be different from $b_{kj}$.  Since these two nodes are distinct, it follows that $(a_{ki}-b_{kj})$ is also a  holomorphic non-vanishing  function on $\Ucal$.

 \item $i\in I_{1}$, that is $x_i \in \hat{C}^1_{\xx}$: let us write $a_{ki}-b_{kj}=t/(a_{ji}-b_{jk})$. By the same argument as above, we conclude that $(a_{ji}-b_{jk})$ is a  holomorphic non-vanishing  function on $\Ucal$.
\end{itemize}
We then have
$$
\frac{\Phi_j}{\Phi_k} =   \frac{(-1)^d t^{-d}}{\prod_{i=1}^n(a_{ki}-b_{kj})^{k_i}}
                      =   (-1)^d t^{-d-\sum_{i\in I_{1}} k_i}\cdot\frac{\prod_{i\in I_{1}}(a_{ji}-b_{jk})^{k_i}}{\prod_{i\in I_{0}}(a_{ki}-b_{kj})^{k_i}}
                      =  (-1)^dt^{d\mu_\Sc} \cdot \frac{\prod_{i\in I_{1}}(a_{ji}-b_{jk})^{k_i}}{\prod_{i\in I_{0}}(a_{ki}-b_{kj})^{k_i}}.
$$
It remains to show that $t^{d\mu_\Sc}=t^{\beta_k}/t^{\beta_j}$. Indeed,  all the nodes $\alpha' \in \NN(\xx)$ different from $\alpha$ do not separate the components $C^j_\xx$ and $C^k_\xx$. Thus we have $\beta_{j,\alpha'}=\beta_{k,\alpha'}$. By definition, $\beta_{k,\alpha}=d\mu_\Sc$, while $\beta_{j,\alpha}=0$, and the claim follows.
As a consequence, we get
$$
\frac{t^{\beta_j}\Phi_j}{t^{\beta_k}\Phi_k}=(-1)^d\frac{\prod_{i\in I_{1}}(a_{ji}-b_{jk})^{k_i}}{\prod_{i\in I_0}(a_{ki}-b_{kj})^{k_i}}=:f_{jk}$$
is a holomorphic non-vanishing function on $\Ucal$.
\end{proof}

\subsection*{Proof of Proposition~\ref{prop:Kaw:trivial:ratio}}
\begin{proof}
Since $dz_j$ is a meromorphic section of the relative canonical line bundle  $K_{\ol{\Cc}_\Uc/\Uc}$, $\Phi_j$ is a  meromorphic section of $\Kc^{(d)}_{0,n|\ol{\Cc}_\Uc}$. It is clear that for all $\xx' \in \Uc^*$, the restriction of $\Phi_j$ to $C_{\xx'}$ is a $d$-differential in $\stratesp$. By definition, $t^{\beta_j}$ does not vanish on $\Uc^*$. Thus the same is true for $t^{\beta_j}\Phi_j$.

We now show that $t^{\beta_j}\Phi_j$ gives a holomorphic section of $\Kc^{(d)}_{0,n|\ol{\Cc}_\Uc}$ for all $j=0,\dots,r$.
Consider a point $(z^0,t^0,u^0) \in \ol{\Cc}_\Uc$.
By definition, $z^0\in (\CP^1)^{r+1}$ is a point in the curve $C(t^0,u^0)$.
By Lemma~\ref{lm:ratio:sect:Kaw:ln:bdl:2:comp},  for all $j,j'\in \{0,\dots,r\}$,  $\displaystyle f_{jj'}:=\frac{t^{\beta_j}\Phi_j}{t^{\beta_{j'}}\Phi_{j'}}$ is a non-vanishing holomorphic function  on $\Uc$.
Therefore, it is enough to show  that there is some $j \in \{0,\dots,r\}$ such that $t^{\beta_j}\Phi_j$ gives a holomorphic section of $\Kc^{(d)}_{0,n}$ near  $(z^0,t^0,u^0)$.
We have two cases:
\begin{itemize}
\item[$\bullet$] $z^0$ a smooth point of $C(t^0,u^0)$. In this case there is a unique component $C^k(t^0,u^0)$ of $C(t^0,u^0)$ such that $z^0 \in C^k(t^0,u^0)$. There is some $j \in \{0,\dots,r\}$ such that the restriction of $\pi_j$ to the component $C^k(t^0,u^0)$ is an isomorphism onto $\CP^1$.
    We claim that for this choice of $j$, $\Phi_j$ (and therefore $t^{\beta_j}\Phi_j$)is holomorphic section of $\Kc^{(d)}_{0,n}$ in a neighborhood of $(z^0,t^0,u^0)$.
Indeed,  in a neighborhood of $(z^0,t^0,u^0)$, $dz_j$ gives a trivializing section of $K_{\ol{\Cc}_\Uc/\Uc}$ , and a trivializing section of $(d-1)\cdot\Gamma_1+\dots+(d-1)\cdot\Gamma_n$ is given by $\prod_{i=1}^n(z_j-a_{ji})^{1-d}$. Since we have $k_i \geq 1-d$, the claim follows. \\

\item[$\bullet$] $z_0$ is a node of $C(t^0,z^0)$ which is the intersection of two components $C^{k}(t^0,u^0)$ and $C^{k'}(t^0,u^0)$.
Let $\alpha=\{j,j'\} \in \NN(\xx)$ be the corresponding node in $C_\xx$. We have $t^0_\alpha=0$, and $(z^0_j,z^0_{j'})=(b_{jj'}(u^0),b_{j'j}(u^0))$.
We can assume that $\pi_{j|C^k(t^0,u^0)}: C^k(t^0,u^0) \to \CP^1$ and $\pi_{j'|C^{k'}(t^0,u^0)}: C^{k'}(t^0,u^0) \to \CP^1$ are isomorphisms.
Note that $\pi_j(C^{k'}(t^0,u^0))=\{b_{jj'}(u^0)\}$ and $\pi_{j'}(C^k(t^0,u^0))=\{b_{j'j}(u^0)\}$.
We can also assume that $C^{j'}_\xx$ is contained in the subcurve $\hat{C}^0_{\xx,\alpha}$ of $C_\xx$, which means that $\beta_{j,\alpha}=0$ and $\beta_{j',\alpha}=d+\sum_{i\in I_{0,\alpha}} k_i \geq 0$.

We claim that $\Phi_j$ gives a holomorphic section of $\Kc^{(d)}_{0,n}$ near $(z^0,t^0,u^0)$.
For all $i \in I_{1,\alpha}$, we have $a_{ji}(t^0,u^0) \neq b_{jj'}(u^0)$. Thus $z_j-a_{ji}$ is an invertible function near $(z^0,t^0,u^0)$.
However for all $i \in I_{0,\alpha}$, we have $a_{ji}(t^0,u^0)=b_{jj'}(u^0)=z^0_j$ (see Lemma~\ref{lm:univ:curv:geom:codim2} (iii)). In particular $z_j-a_{ji}$ is not invertible near $(z^0,t^0,u^0)$.
Using \eqref{eq:rel:coord:through:node}, we can write
$$
z_j-a_{ji}=-t_\alpha\frac{z_{j'}-a_{j'i}}{(z_{j'}-b_{j'j})(a_{j'i}-b_{j'j})}=-(z_j-b_{jj'})\cdot\frac{z_{j'}-a_{j'i}}{a_{j'i}-b_{j'j}}
$$
where the functions $(z_{j'}-a_{j'i})$ and $(a_{j'i}-b_{j'j})$ are invertible near $(z^0,t^0,u^0)$. It follows that
\begin{equation}\label{eq:phi:near:node}
\Phi_j=(-1)^{\sum_{i\in I_{0,\alpha}}k_i}(z_j-b_{jj'})^{d+\sum_{i\in I_{0,\alpha}}k_i}\cdot \prod_{i\in I_{1,\alpha}}(z_j-a_{ji})^{k_i} \cdot \prod_{i\in I_{0,\alpha}}\left(\frac{z_{j'}-a_{j'i}}{a_{j'i}-b_{j'j}}\right)^{k_i} \cdot \left( \frac{dz_j}{z_j-b_{jj'}} \right)^d.
\end{equation}
Now, since $\frac{dz_j}{z_j-b_{jj'}}$ gives a trivializing section of $K_{\ol{\Ccal}_\Uc/\Ucal}$ near the point $(z^0,t^0,u^0)$, and $d+\sum_{i\in I_{0,\alpha}} \geq 0$, the claim follows. This concludes the proof of the proposition.
\end{itemize}
\end{proof}
\subsection{Example}\label{subsec:ex:sect:codim2}
To close this section, we give the explicit formulas for the sections $t^{\beta_j}\Phi_j$'s near a point $\xx$ in a stratum of codimension $2$ of $\ol{\Mod}_{0,n}$.
In this case, $C_\xx$ is a stable curve with $3$ irreducible components denoted by $C^0_\xx,C^1_\xx,C^2_\xx$ with $C_\xx^0$ intersecting both $C^1_\xx$ and $C^2_\xx$.
Let $\alpha_j$ denote the node between $C^0_\xx$ and $C^j_\xx, \; j=1,2$.
Let $\{I_0,I_1,I_2\}$ be the partition of $\{1,\dots,n\}$ where $I_j$ records the marked points in the component $C^j_\xx$.
Set $n_j=|I_j|, \; j=0,1,2$.
We can assume that $I_0=\{1,\dots,n_0\}, I_1=\{n_0+1,\dots,n_0+n_1\}$, and $I_2=\{n_0+n_1+1,\dots,n\}$.
Since $C_\xx$ is stable, we must have $n_0\geq 1, n_1\geq 2$, and $n_2\geq 2$.

Recall that each component of $C_\xx$ is isomorphic to $\CP^1$.
Let $x,y,z$ be the inhomogeneous coordinates on $C^0_\xx, C^1_\xx,$ and $C^2_\xx$ respectively.
Let $a_i$ be the (inhomogeneous) coordinate of the $i$-th marked point of $C_\xx$ in $C^0_\xx$ for $i=1,\dots,n_0$, $b_i$ the coordinate of the $i$-th marked point in $C^1_\xx$ for $i=n_0+1,\dots,n_0+n_1$, and $c_i$ the coordinate of the $i$-marked point in $C^2_\xx$ for $i=n_0+n_1+1,\dots,n$.
Using the automorphisms of $\CP^1$, we can assume that
\begin{itemize}
\item In $C^0_\xx$: $0$ corresponds to the node $\alpha_1$, $1$ corresponds to the node $\alpha_2$, and $a_{n_0}=\infty$.
\item In $C^1_\xx$: $0$ corresponds to the node $\alpha_1$, $b_{n_0+n_1-1}=1, b_{n_0+n_1}=\infty$.
\item In $C^2_\xx$: $0$ corresponds to the node $\alpha_2$, $c_{n-1}=1, c_n=\infty$.
\end{itemize}
Let $U_0$ be an open neighborhood of $(a_i)_{1\leq i \leq n_0-1}$ in $\C^{n_0-1}$, $U_1$ an open neighborhood of $(b_{n_0+i})_{1\leq i \leq n_1-2}$ in $\C^{n_1-2}$, and $U_2$ an open neighborhood of $(c_{n_0+n_1+i})_{1\leq i \leq n_2-2}$ in $\C^{n_2-2}$.

Let $\Delta\subset \C$ be a small disk about $0$. A neighborhood of $\xx$ in $\ol{\Mod}_{0,n}$ can be identified with $\Ucal=\Delta^2\times U_0\times U_1\times U_2$.
Let $(t_1,t_2)$ be the coordinates on $\Delta^2$, where $t_i$ corresponds to the  node $\alpha_i$.
For all $(t_1,t_2) \in \Delta^2$, define
$$
C_{(t_1,t_2)}:=\{(x,y,z) \in \left(\CP^1\right)^3, \; xy=t_1, \; (x-1)z=t_2\} \subset \left( \CP^1\right)^3.
$$
The universal curve over $\Ucal$ is isomorphic to
$$
\ol{\Ccal}_{\Ucal}:=\{(x,y,z,t_1,t_2,X,Y,Z) \in \left(\CP^1\right)^3\times \Delta^2\times U_0\times U_1 \times U_2, \; xy=t_1, (x-1)z=t_2\}
$$
where $X=(x_i)_{1\leq i\leq n_0-1} \in U_0, Y=(y_i)_{n_0+1\leq i \leq n_0+n_1-2} \in U_1$, and $Z=(z_i)_{n_0+n_1+1\leq i \leq n-2} \in U_2$.
We set
$$
x_{n_0}=\infty, \ y_{n_0+n_1-1}=1, \ y_{n_0+n_1}=\infty, \ z_{n-1}=1, z_n=\infty,
$$
and define
\begin{eqnarray*}
\sigma_i(t_1,t_2,X,Y,Z) &:=& (x_i, \frac{t_1}{x_i}, \frac{t_2}{x_i-1}) \in C_{(t_1,t_2)}, \, i=1\dots,n_0,\\
\sigma_i(t_1,t_2,X,Y,Z) &:=& (\frac{t_1}{y_i},y_i,\frac{t_2y_i}{t_1-y_i}) \in C_{(t_1,t_2)},\, i=n_0+1,\dots,n_0+n_1,\\
\sigma_i(t_1,t_2,X,Y,Z) &:=& (\frac{t_2+z_i}{z_i},\frac{t_1z_i}{t_2+z_i},z_i) \in C_{(t_1,t_2)}, \, i=n_0+n_1+1,\dots,n.
\end{eqnarray*}
The pointed curve represented by $(t_1,t_2,X,Y,Z)$ is isomorphic to $(C_{(t_1,t_2)},(\sigma_i(t_1,t_2,X,Y,Z))_{1\leq i\leq n})$.
By construction, we have
\begin{eqnarray*}
 \Phi_0 & = & \prod_{i\in I_0}(x-x_i)^{k_i} \prod_{i\in I_1}(x-\frac{t_1}{y_i})^{k_i}\prod_{i\in I_2}(x-\frac{t_2+z_i}{z_i})^{k_i}(dx)^d,\\
 \Phi_1 & = & \prod_{i\in I_0}(y-\frac{t_1}{x_i})^{k_i} \prod_{i\in I_1}(y-y_i)^{k_i} \prod_{i \in I_2}(y-\frac{t_1z_i}{t_2+z_i})^{k_i}(dy)^d,\\
 \Phi_2 & = & \prod_{i\in I_0}(z-\frac{t_2}{x_i-1})^{k_i} \prod_{i\in I_1}(z-\frac{t_2y_i}{t_1-y_i})^{k_i} \prod_{i \in I_2}(z-z_i)^{k_i}(dz)^d.
\end{eqnarray*}
Set $\hat{k}_j:=\sum_{i\in I_j} k_i, \; j=0,1,2$, and $\nu_1:=-d-\hat{k}_1 = d+\hat{k}_0+\hat{k}_2, \; \nu_2:=-d-\hat{k}_2 = d+\hat{k}_0+\hat{k}_1$.
Using the relations $y=\frac{t_1}{x}$, and $z=\frac{t_2}{x-1}$, one can readily check that
\begin{eqnarray*}
\frac{\Phi_1}{\Phi_0} & = & (-1)^d\frac{\prod_{i\in I_1}y_i^{k_i}}{\prod_{i\in I_0}x_i^{k_i}}\cdot\prod_{i\in I_2}\left(\frac{z_i}{t_2+z_i}\right)^{k_i} t_1^{\nu_1},\\
\frac{\Phi_2}{\Phi_0}& = & (-1)^d \prod_{i\in I_1}\left(\frac{y_i}{t_1-y_i}\right)^{k_i}\frac{\prod_{i\in I_2}z_i^{k_i}}{\prod_{i\in I_0}(x_i-1)^{k_i}} t_2^{\nu_2}.
\end{eqnarray*}
We have the following possibilities:
\begin{itemize}
 \item[(a)] $\nu_1 \leq 0$ and  $\nu_2 \leq 0$: in this case $t^{\beta_0}=1, t^{\beta_1}=t_1^{-\nu_1}, t^{\beta_2}=t_2^{-\nu_2}$,
$$
\frac{t^{\beta_1}\Phi_1}{\Phi_0}  =  (-1)^d\frac{\prod_{i\in I_1}y_i^{k_i}}{\prod_{i\in I_0}x_i^{k_i}}\cdot\prod_{i\in I_2}\left(\frac{z_i}{t_2+z_i}\right)^{k_i}, \quad \frac{t^{\beta_2}\Phi_2}{\Phi_0} =  (-1)^d \prod_{i\in I_1}\left(\frac{y_i}{t_1-y_i}\right)^{k_i}\frac{\prod_{i\in I_2}z_i^{k_i}}{\prod_{i\in I_0}(x_i-1)^{k_i}}
$$
are non-vanishing holomorphic functions on $\Uc$, and $\Phi_0 \sim t_1^{-\nu_1}\Phi_1 \sim t_2^{-\nu_2}\Phi_2$ are holomorphic sections of $\Kc^{(d)}_{0,n}$ over $\ol{\Cc}_\Uc$.
Note that the restriction of $\Phi_0$ to $C^0_\xx$ gives a non-zero $d$-differential, while its restrictions to $C^1_\xx$ and $C^2_\xx$ vanish identically. This means that $\Phi_0$ induces a section of the bundle $\Pb\Hc^{0,n}$ over $\Uc$.

 \item[(b)] $\nu_1\cdot\nu_2\leq 0$: without loss of generality, let us assume that $\nu_1\geq 0$ and $\nu_2 \leq 0$: in this case $\beta_0=(\nu_1,0),\beta_1=(0,0),\beta_2=(\nu_1,-\nu_2)$, and  $t_1^{\nu_1}\Phi_0 \sim \Phi_1 \sim t_1^{\nu_1}t_2^{-\nu_2}\Phi_2$ are holomorphic sections of $\Kc^{(d)}_{0,n}$ over $\ol{\Cc}_\Uc$. Since the restriction of $\Phi_1$ to $C^1_\xx$ is non-zero, $\Phi_1$ induces a section of the bundle $\Pb\Hc^{(d)}_{0,n}$ over $\Uc$.

\item[(c)] $\nu_1\geq 0$ and $\nu_2 \geq 0$: in this case, $\beta_{0}=(\nu_1,\nu_2)$, $\beta_1=(0,\nu_2)$, $\beta_2=(\nu_1,0)$, and $t_1^{\nu_1}t_2^{\nu_2}\Phi_0 \sim t_2^{\nu_2}\Phi_1 \sim t_1^{\nu_1}\Phi_2$ are all holomorphic sections of $\Kc^{(d)}_{0,n}$ on $\ol{\Cc}_\Uc$.  However, those sections vanish identically on all of the components of the central fiber $C_\xx$. Therefore, we do not get a section of $\Pb\Hc^{(d)}_{0,n}$ over $\Uc$.
\end{itemize}
\section{Blow-up of $\ol{\Mod}_{0,n}$}\label{sec:blowup:M0n}
We have seen in \textsection\ref{subsec:ex:sect:codim2} that the section $\ff: \Mod_{0,n} \to \Pb\Hc^{(d)}_{0,n}$ does not extend to $\ol{\Mod}_{0,n}$ in general. The points at which the natural extension fails are exactly where the all the monomials $t^{\beta_j}$ vanish. This prompts us to consider the blow up of  $\ol{\Mod}_{0,n}$ along the ideal generated by the $t^\beta_j$'s.

\subsection{Construction}\label{subsec:blowup:construct}
Let $\xx$ be a point in a stratum of codimension $r$ in $\ol{\Mod}_{0,n}$ as in \textsection\ref{subsec:univ:curv:near:bdry}.
Recall that $\NN(\xx)$ is the set of nodes of $C_\xx$.
By definition,  $\xx$ is contained in the intersection $\cap_{\alpha\in \NN(\xx)} D_{\Sc_\alpha}$, where $\Sc_\alpha$ is the partition corresponding to the node $\alpha$. Let $\Ucal_\xx$ be a neighborhood of $\xx$ that satisfies the following
\begin{itemize}
 \item[(i)] $\Ucal_\xx$ does not intersect any boundary divisor $D_\Scal\notin \{D_{\Sc_\alpha}, \, \alpha \in \NN(\xx)\}$,

 \item[(ii)] $\Ucal_\xx$ can be identified with an open subset of $\C^{n-3}$ such that, for each $\alpha \in \NN(\xx)$, there is a coordinate function $t_\alpha$ such that $D_{\Sc_\alpha}\cap \Ucal_\xx$ is defined by the equation $t_\alpha=0$.
\end{itemize}
We can actually choose $\Ucal_\xx$ to be an open affine of $\ol{\Mod}_{0,n}$ such that $t_\alpha$ are elements of the coordinate ring of $\Ucal_\xx$.
To simplify the notations, throughout this section, we will write $D_\alpha$ instead of $D_{\Sc_\alpha}$, and $\mu_\alpha$ instead of $\mu_{\Sc_\alpha}$.
Let  $\Ocal_{\ol{\Mod}_{0,n}}$ be the structure sheaf of $\ol{\Mod}_{0,n}$, and $\Ical_{\Ucal_\xx}$ be the  sheaf of ideals of $\Ocal_{\ol{\Mod}_{0,n}|\Ucal_\xx}$ generated by $\{t^{\beta_0},\dots,t^{\beta_r}\}$.
\begin{Theorem}\label{th:ideal:sheaf:def}
 The sheaves of ideals $\{\Ical_{\Ucal_\xx},\xx \in \ol{\Mod}_{0,n}\}$ patch together and give rise to an ideal  sheaf $\Ical$  of $\Ocal_{\ol{\Mod}_{0,n}}$ whose support is contained in $\partial \ol{\Mod}_{0,n}$.
\end{Theorem}
\begin{proof}
We need to show that $\Ical_{\Ucal_\xx|\Ucal_\xx\cap\Ucal_{\xx'}} \simeq \Ical_{\Ucal_{\xx'}|\Ucal_\xx\cap\Ucal_{\xx'}}$ for all $\xx,\xx' \in \ol{\Mod}_{0,n}$.
We first consider the case $\xx' \in \Ucal_\xx$. In this case, all the boundary divisors that contain $\xx'$ contain also $\xx$. Thus $\NN(\xx')$ is a subset of $\NN(\xx)$.
Let $s=|\NN(\xx')|$. The curve $C_{\xx'}$ has $(s+1)$ irreducible components, which will be denoted by $C^0_{\xx'},\dots,C^s_{\xx'}$.
For all $\alpha\in \NN(\xx')$, let $t'_\alpha$ be a function in the affine coordinate ring of $\Ucal_{\xx'}$ that defines  $D_\alpha\cap\Ucal_{\xx'}$.
Let $t'=(t'_\alpha)_{\alpha\in \NN(\xx')}$.
For $j'=0,\dots,s$, define the exponent vectors ${\beta'_{j'}}=(\beta'_{j',\alpha})_{\alpha\in \NN(\xx')}$, and the monomial ${t'}^{\beta'_{j'}}$  in the same way as $\beta_j$ and $t^{\beta_j}$.
By a slight abuse of notation, we consider $\beta'_j$ as a vector in $\Z_{\geq 0}^{\NN(\xx)}$ by setting $\beta'_{j,\alpha}=0$ for all $\alpha \in \NN(\xx)\setminus \NN(\xx')$.

By definition $\Ical_{\Ucal_{\xx'}}$ is generated by $({t'}^{\beta'_0},\dots,{t'}^{\beta'_s})$.
Since the restrictions of $t'_\alpha$ and $t_\alpha$ to $\Ucal_\xx\cap\Ucal_{\xx'}$ define the same divisor, $t'_\alpha/t_\alpha$ is an invertible regular function on $\Ucal_\xx\cap \Ucal_{\xx'}$.
It follows that $\Ical_{\Ucal_{\xx'}|\Ucal_\xx\cap\Ucal_{\xx'}}=(t^{\beta'_0},\dots,t^{\beta'_s})$.

The curve $C_{\xx'}$ can be obtained from $C_\xx$ by smoothening the nodes in $\NN(\xx)\setminus \NN(\xx')$. Therefore, there is a continuous map $\varphi: C_{\xx'} \ra C_\xx$ that sends each component of $C_{\xx'}$ onto a subcurve  of $C_\xx$.
Consider an irreducible component $C^i_\xx$ of $C_\xx$.
There is an irreducible component $C^j_{\xx'}$ such that $C^i_\xx\subset \varphi(C^j_{\xx'})$.
By construction, for all $\alpha\in \NN(\xx')$, we have $\beta_{i,\alpha}=\beta'_{j,\alpha}$.
For all $\alpha \in \NN(\xx)\setminus N(\xx')$, $t_\alpha$ is an invertible regular function on $\Ucal_\xx\cap\Ucal_{\xx'}$ (because $\xx'$ is not contained in the divisor $D_\alpha$).
Hence $t^{\beta_i}/t^{\beta'_j}$ is an invertible regular function on $\Ucal_\xx\cap\Ucal_{\xx'}$. It follows that
$$
\Ical_{\Ucal_{\xx'}|\Ucal_\xx\cap\Ucal_{\xx'}}\simeq (t^{\beta'_0},\dots,t^{\beta'_s})\simeq (t^{\beta_0},\dots,t^{\beta_r})\simeq \Ical_{\Ucal_{\xx}|\Ucal_\xx\cap\Ucal_{\xx'}}.
$$

Consider now the general case. For every $\xx'' \in \Ucal_{\xx}\cap\Ucal_{\xx'}$, we consider the open affine $\Ucal_{\xx''}$ of $\xx''$ which satisfies (i) and (ii).
The above arguments show that we have
$$
\Ical_{\Ucal_\xx|\Ucal_\xx\cap\Ucal_{\xx'}\cap\Ucal_{\xx''}} \simeq \Ical_{\Ucal_{\xx''}|\Ucal_\xx\cap\Ucal_{\xx'}\cap\Ucal_{\xx''}}\simeq \Ical_{\Ucal_{\xx'}|\Ucal_\xx\cap\Ucal_{\xx'}\cap\Ucal_{\xx''}}.
$$
It follows that $\Ical_{\Ucal_\xx|\Ucal_\xx\cap\Ucal_{\xx'}} \simeq \Ical_{\Ucal_{\xx'}|\Ucal_\xx\cap\Ucal_{\xx'}}$. Hence $\Ic$ is a well defined ideal sheaf of $\Ocal_{\ol{\Mod}_{0,n}}$. It is also clear from the definition that if $\xx\in \Mod_{0,n}$ then $\Ic_\xx \simeq \Ocal_{\ol{\Mod}_{0,n},\xx}$. Therefore $\mathrm{supp}(\Ic) \subset \partial\ol{\Mod}_{0,n}$.
\end{proof}

\begin{Definition}\label{def:blowup}
We denote by $\blowupsp$  the blow-up of $\ol{\Mod}_{0,n}$ with respect to the sheaf of ideals $\Ical$ (see \cite[p.163]{Hart} or \cite[\textsection IV.2]{EH00}).
\end{Definition}

By definition $\blowupsp$ comes equipped with a projection $\hat{p}: \blowupsp \ra \ol{\Mod}_{0,n}$.
As an analytic space, $\blowupsp$ can be described as follows: let $u=(u_1,\dots,u_{n-3-r})$ be a family of holomorphic functions  such that $(t,u)$ is a coordinate system on $\Uc_\xx$.
Let $\Ucal^*_\xx$ denote the set $\Ucal_\xx\cap \Mod_{0,n}$, that is $\Ucal^*_\xx=\{(t,u) \in \Ucal_\xx, \; t_\alpha \neq 0, \, \forall \alpha\in \NN(\xx)\}$.
Assuming that $r>0$, let
$$
\widehat{\Ucal}^{*}_\xx:=\{(t,u,[t^{\beta_0}:\dots:t^{\beta_r}]), \; (t,u) \in \Ucal^*_\xx\} \subset \Ucal^*_\xx\times \CP^r,
$$
and
$$
\Ucal^\#_\xx=\{(t,u,[v_0:\dots:v_r]) \in \Ucal_\xx\times \CP^{r}, \; v_jt^{\beta_k}=v_kt^{\beta_j}, \, \forall j,k \in \{0,\dots,r\}, j\neq k \} \subset \Ucal_\xx\times \CP^r.
$$
Then $\widehat{\Ucal}_\xx:=\hat{p}^{-1}(\Ucal_\xx)$ is isomorphic to the irreducible component of  $\Ucal^\#_\xx$ that contains $\widehat{\Ucal}^{*}_\xx$. Equivalently, one can define $\widehat{\Ucal}_\xx$ to be the closure of $\widehat{\Ucal}^*_\xx$ in $\Ucal\times\CP^r$.

\begin{Remark}\label{rk:blowup:independ:gen:set}
It is worth noticing that  $\widehat{\Ucal}_{\xx}$ depends only on $\Ical_{\Ucal_\xx}$, which means that if in the construction of $\widehat{\Ucal}_\xx$ we take two different generating sets of $\Ical_{\Ucal_\xx}$, then the resulting varieties are isomorphic.
\end{Remark}

We have the following classical result (see for instance \cite[Ch.II.\textsection 7, Prop. 7.16]{Hart}).
\begin{Proposition}\label{prop:proj:variety}
 The space $\blowupsp$ is a complex projective variety, and the natural projection $\hat{p}: \blowupsp \ra \ol{\Mod}_{0,n}$ is a birational, proper,  surjective morphism.
\end{Proposition}

\subsection{Support of the ideal sheaf $\Ical$}\label{subsec:supp:I}
Consider a point $\xx\in \ol{\Mod}_{0,n}$ as in \textsection\ref{subsec:blowup:construct}.
Recall that the node of $C_\xx$ corresponding to the pair $\alpha=\{j,j'\} \in \NN(\xx)$ is obtained by identifying the point $y_{jj'} \in C^j_\xx$ and the point $y_{j'j}\in C^{j'}_\xx$.
We assign to $y_{jj'}$ and $y_{j'j}$  some weights as follows: let $\hat{C}^0_{\xx,\alpha}$ and $\hat{C}^1_{\xx,\alpha}$ be the subcurves of $C_\xx$ that are separated by $\alpha$, we can assume that $C^j_\xx \subset \hat{C}^{0}_{\xx,\alpha}$, and $C^{j'}_\xx \subset \hat{C}^{1}_{\xx,\alpha}$, then
\begin{equation}\label{eq:def:weight:node}
\mu(y_{jj'})=\sum_{i\in I_{1,\alpha}} \mu_i \quad \text{ and } \quad \mu(y_{j'j})=\sum_{i \in I_{0,\alpha}} \mu_i.
\end{equation}
We have
\begin{equation}\label{eq:sum:weight:node}
   \mu(y_{jj'})+\mu(y_{j'j})=2, \quad \hbox{ for all  $\{j,j'\} \in \NN(\xx)$},
\end{equation}
and
\begin{equation}\label{eq:sum:weight:comp}
 \sum_{i \in I_j } \mu_i +\sum_{j'\in \NN_j(\xx)} \mu(y_{jj'}) =2, \quad \hbox{for each $j \in \{0,\dots,r\}$},
\end{equation}
where $\NN_j(\xx)=\{j' \in \{0,\dots,r\}, \; \{j,j'\} \in \NN(\xx)\}$ is the set  of nodes meeting the component $C^j_\xx$.

Each node of $C_\xx$ corresponds to a geometric edge of the dual tree $\T_\xx$. Such an edge consists of a pair of  oriented edges with opposite orientations.
The oriented edge from $v_j$ to $v_{j'}$ will be denoted by $e_{jj'}$.
We associate to $e_{jj'}$  the weight $\mu(e_{jj'})$ defined by:
$$
\mu(e_{jj'}):=\mu(y_{j'j})-\mu(y_{jj'}).
$$
Note that if $\alpha=\{j,j'\} \in \NN(\xx)$ then
\begin{equation}\label{eq:weight:node}
\mu_{\alpha}=\frac{1}{2}(\sum_{i\in I_{1,\alpha}}\mu_i-\sum_{i\in I_{0,\alpha}}\mu_i)=\frac{1}{2}|\mu (e_{jj'})|.
\end{equation}
We will call $\mu_{\alpha}$ the weight of $\alpha$.

\begin{Definition}[principal component and principal subcurve]\label{def:princ:comp}
Assume that all the nodes in $C_\xx$ have strictly positive weights. Then a component $C_\xx^j$ of $C_\xx$ is said to be {\em principal} if  we have $\mu(e_{jj'}) >0$(or equivalently  $\mu(y_{jj'}) < 1$) for all $j' \in \NN_j(\xx)$.

In the case  some nodes in $C_\xx$ have zero weight, we first smoothen those nodes to obtain another curve $C_{\xx'}$ parametrized by a point $\xx'$ in a neighborhood of $\xx$.
Each irreducible component of $C_{\xx'}$ corresponds to a subcurve of $C_\xx$ which is the union of some components connected by nodes of zero weight.
A subcurve of $C_\xx$  that corresponds to a principal component of $C_{\xx'}$ is called a {\em principal subcurve} of $C_\xx$. A component of $C_\xx$ is said to be principal if it is contained in a principal subcurve.
\end{Definition}

We will prove
\begin{Proposition}\label{prop:supp:I:princ:comp}
A point $\xx \in \ol{\Mod}_{0,n}$ is {\bf not} contained in the support of the sheaf of ideals $\Ical$ if and only  if $C_\xx$ has a unique principal subcurve.
\end{Proposition}

We first have
\begin{Lemma}\label{lm:princ:subcurv:weight}
If $C^j_\xx$ and $C^{j'}_\xx$ are contained in the same principal subcurve, then $\beta_j=\beta_{j'}$.
\end{Lemma}
\begin{proof}
By definition, if the node $\alpha$ does not separate $C^j_\xx$ and $C^{j'}_\xx$, then $\beta_{j,\alpha}=\beta_{j',\alpha}$.
Since by definition all the edges that separate $C^j_\xx$ and $C^{j'}_\xx$ have weight zero, the lemma follows.
\end{proof}

\begin{Lemma}\label{lm:princ:comp:exist}
There always exists at least one principal component.
\end{Lemma}
\begin{proof}
It suffices to consider the case $C_\xx$ does not have zero weighted nodes.
If there exists  no principal  component, then for every vertex $v_j$ of $\T_\xx$ there exists an oriented edge $e_{jj'}$ such that $\mu(e_{jj'})<0$.  Consider a path $\gamma$ in $\T_\xx$ that is composed uniquely by oriented edges with negative weight. Let $e_{jj'}$ be the last edge of this path. Since there exists $j''$ such that $\mu(e_{j'j''}) <0$ (note that we must have $j''\neq j$), we can always extend $\gamma$ to a longer path. Thus, there exists such a path $\gamma$ of length $r+1$. Since $\T_\xx$ has exactly $r+1$ vertices, this means that $\gamma$ must contain  a cycle. But $\T_\xx$ is a tree, so this is impossible.
\end{proof}


\begin{Lemma}\label{lm:comp:connect:princ:comp}
If the component $C^j_\xx$ of $C_\xx$ is  not principal, then there is a principal component $C^{k}_\xx$ such that the unique path in $\T_\xx$ from $v_j$ to $v_{k}$ consists  only of  edges of weight $\leq 0$, and at least one of those edges has negative weight.
\end{Lemma}
\begin{proof}
Again, it is enough to prove the lemma in the case $C_\xx$ does not have zero weighted nodes.
Assume that $C^{j}_\xx$ is not principal. Set $j_0=j$.  By definition, $C^{j_0}_\xx$ has an adjacent component $C^{j_1}_\xx$ such that the edge $e_{j_0j_1}$ is negatively  weighted.
 If $C^{j_1}_\xx$ is not principal then there exists $C^{j_2}_\xx$ such that $\mu(e_{j_1j_2})<0$. Note that $\mu(e_{j_1j_0})=-\mu(e_{j_0j_1})>0$, therefore $e_{j_1j_2} \neq e_{j_1j_0}$.
 Continue this procedure we get a sequence of negatively weighted edges $(e_{j_0j_1},e_{j_1j_2},\dots,e_{j_{i-1}j_{i}})$ of $\T_\xx$.
 The concatenation of those edges gives a path $\gamma$ in $\T_\xx$.
 If $C^{j_i}_\xx$ is not principal, then the path $\gamma$ can be extended to a longer path.
 Since $\T_\xx$ is a tree, the length of $\gamma$ cannot be larger than $r$, which means that the procedure must terminate after finitely many steps.
 Let $s$ be the length of  $\gamma$ when the procedure terminates.
 Then $C^{j_s}_\xx$ is a principal component, and  $\gamma=e_{j_{s-1}j_s}*\dots*e_{j_0j_1}$ is the desired path.
\end{proof}

As a consequence of Lemma~\ref{lm:comp:connect:princ:comp}, we have

\begin{Lemma}\label{lm:blowup:non:princ:comp}
If $C^j_\xx$ is not a principal component of $C_\xx$, then there exists $k\in \{0,\dots,r\}$ such that $t^{\beta_j}/t^{\beta_k}$ is a non-constant monomial vanishing at $0\in \Delta^{\NN(\xx)}$.
\end{Lemma}
\begin{proof}
By Lemma~\ref{lm:comp:connect:princ:comp}, there is a principal component $C^k_\xx$ such that the path from $v_j$ to $v_k$ in $\T_\xx$ consists only of negatively weighted edges.
This means that for any node $\alpha$ that separates $C^j_\xx$ and $C^k_\xx$, we always have $C^j_\xx\subset \hat{C}^{0}_{\xx,\alpha}$.
Thus, $t^{\beta_j}/t^{\beta_k}=\prod_{\alpha\in \fracs(j,k)}t_\alpha^{d\mu_\alpha}$, where $\fracs(j,k)$ is the set of nodes that separate $C^j_\xx$ and $C^k_\xx$.
\end{proof}

Let  $\mathring{C}^1_\xx,\dots,\mathring{C}^{r_0}_\xx$ denote the principal subcurves of $C_\xx$.
Let $\beta^*_m$ be the exponent vector associated with any irreducible component of $\mathring{C}^m_\xx$ (by Lemma~\ref{lm:princ:subcurv:weight} all the components of $\mathring{C}^m_\xx$ have the same associated exponent vector).
An immediate consequence of Lemma~\ref{lm:blowup:non:princ:comp} is
\begin{Corollary}\label{cor:generators:I}
  The ideal sheaf $\Ical_{|\Ucal_\xx}$ is generated by the functions $\{t^{\beta^*_1},\dots,t^{\beta^*_{r_0}}\}$.
\end{Corollary}

\subsection*{Proof of Proposition~\ref{prop:supp:I:princ:comp}}
\begin{proof}
Recall that $\Ical$ is generated by the functions $\{t^{\beta_0},\dots,t^{\beta_r}\}$ in a neighborhood of $\xx$.
Assume first that $C_\xx$ has only one principal subcurve. We claim that  if $C^k_\xx$ is a principal irreducible component, then $t^{\beta_k}=1$, which implies that $\Ical \simeq \Ocal_{\ol{\Mod}_{0,n}}$ in a neighborhood of $\xx$. Suppose that $t^{\beta_k}\neq 1$. Then there is a node $\alpha_0=\{j_0,j'_0\}$ such that $C^k_\xx \subset \hat{C}^{0}_{\xx,\alpha_0}$.
Without loss of generality, we can assume that $\alpha_0$ separates $C^{j_0}_\xx$ and $C^k_\xx$.
This means that the unique path in $\T_\xx$ from $v_{j_0}$ to $v_k$ contains the edge $e_{j_0j'_0}$, and $\mu(e_{j_0j'_0})>0$.
It follows in particular that $C^{j_0}_\xx$ is not  principal, because otherwise we would have two principal components that are not connected by a sequence of nodes of weight zero contradicting the assumption that $C_\xx$ has only one principal subcurve.

Since $C^{j_0}_\xx$ is not principal, by Lemma~\ref{lm:comp:connect:princ:comp}, there is a principal component $C^j_\xx$ such that the path from $v_{j_0}$ to $v_j$ consists only of negatively weighted edges.
Since $\mu(e_{j_0j'_0})>0$, the edge $e_{j_0j'_0}$ is not contained in this path.
Since $\T_\xx$ is a tree, the unique path in $\T_\xx$ from $v_k$ to $v_j$ is the concatenation of the path from $v_k$ to $v_{j_0}$ and the path from $v_{j_0}$ to $v_j$.  In particular, the path from $v_k$ to $v_j$ contains $e_{j'_0j_0}$, and again we have a contradiction to the assumption that $C_\xx$ has only one principal component.
Therefore, we can conclude that $t^{\beta_k}=1$.

\medskip

Assume now that $C_\xx$ has at least two principal subcurves. We claim that $t^{\beta_j}\neq 1$ for all $j\in \{0,\dots,r\}$, which implies that $\xx \in \mathrm{supp}(\Ical)$. By Lemma~\ref{lm:blowup:non:princ:comp}, we know that if $C^j_\xx$ is not a principal component, then there exists a principal component $C^k_\xx$ such that $t^{\beta_j}/t^{\beta_k}$ is a non constant monomial in the variables $(t_\alpha)_{\alpha \in \NN(\xx)}$. Thus $t^{\beta_j}\neq 1$ in this case. If $C^j_\xx$ is a principal component, since there are more than one principal subcurve, there is a principal component $C^k_\xx$ and a node $\alpha=\{k,k'\}$ attached to $C^k_\xx$ which separates $C^{k}_\xx$ and $C^j_\xx$. Since $C^k_\xx$ is principal, $\mu(y_{kk'}) < 1$, where $y_{kk'}$ is the point in $C^k_\xx$ corresponding to the node $\alpha$.
This implies that $C^j_\xx \subset \hat{C}^{0}_{\xx,\alpha}$.
Consequently, $\beta_{j,\alpha}=d\mu_\alpha > 0$, and  $t^{\beta_j} \neq 1$, since it contains the factor $t_\alpha^{d\mu_\alpha}$.
The proof of the proposition is now complete.
\end{proof}

Proposition~\ref{prop:supp:I:princ:comp} implies
\begin{Corollary}\label{cor:1:princ:comp}
If for all $\xx \in \ol{\Mod}_{0,n}$ the curve $C_\xx$ has only one principal subcurve, then the blow-up $\blowupsp$ is isomorphic to $\ol{\Mod}_{0,n}$.
\end{Corollary}

\begin{Remark}\label{rk:1:princ:comp}
It is shown in \cite[\textsection 5]{KN18} that if all the weights $\{\mu_i, \, i=1,\dots,n\}$ are positive then $C_\xx$ has a unique principal subcurve for all $\xx \in \ol{\Mod}_{0,n}$.
However, this positivity condition is not necessary. For instance, for $\mu=(\frac{1}{2}^5,\frac{-1}{2})$ or equivalently $\kappa=(-1^5,1)$, we also have $\widehat{\Mod}_{0,6}(-1^5,1) \simeq \ol{\Mod}_{0,6}$.
\end{Remark}


\section{Boundary of the blow-up}\label{sec:bdry:blowup}
Let $\partial \blowupsp:=\hat{p}^{-1}(\partial \ol{\Mod}_{0,n})$.
We will call $\partial\blowupsp$ the boundary of $\blowupsp$.
In this section, we investigate the irreducible components of $\partial\blowupsp$.
Our goal in this section is to prove
\begin{Theorem}\label{th:part:bdry:div:blowup}
Each ireducible component of $\partial\blowupsp$ is a divisor.
The set of irreducible components of $\partial \blowupsp$ is in bijection with the set $\hat{\Pc}(\mu)$ of partitions $\Scal=\{I_0,I_1,\dots,I_r\}$ of $\{1,\dots,n\}$ with $r\geq 1$ such that either

\begin{itemize}
\item[(a)] $r=1$ and $\Sc=\{I_0,I_1\} \in \Pcal$, or

\item[(b)] $r\geq 2$ and $\Scal$ satisfies $\mu(I_0)<1$, and $\mu(I_j)>1$ for all $j=1,\dots,r$.
\end{itemize}
More precisely,  every partition $\Sc$ in  $\hat{\Pc}(\mu)$ determines  a unique stratum $D^*_\Sc$ in $\ol{\Mod}_{0,n}$, and there is a unique irreducible component $\hat{D}_\Sc$ of $\partial \blowupsp$ such that $\hat{D}=\ol{\hat{p}^{-1}(D^*_\Sc)}$.
\end{Theorem}
\begin{Remark}\label{rk:boundary:comp:prop:transf}\hfill
\begin{itemize}
\item[$\bullet$] If $\Scal\in \Pcal$ then  $\hat{D}_\Sc$ is the proper transform of $D_\Sc$ (that is the divisor of $\partial\ol{\Mod}_{0,n}$ associated to $\Sc$).
\item[$\bullet$] In general, each $\Sc=\{I_0,I_1,\dots,I_r\} \in \hat{\Pc}(\mu)$ determines a unique stratum $D^*_\Sc$ in $\partial \ol{\Mcal}_{0,n}$ as follows: any  pointed curve $(C,x_1,\dots,x_n)$ in $D^*_\Sc$ has exactly $r+1$ irreducible components denoted by $C^0,\dots,C^r$, where the component $C^j$ contains the marked points $\{x_i, \; i \in I_j\}$, and for all $j=1,\dots,r$, there is a node between $C^0$ and $C^j$. 
\end{itemize}
\end{Remark}

\begin{Definition}\label{def:part:div:blowup}
The divisor of $\blowupsp$ associated to $\Scal \in \hat{\Pc}(\mu)$ is denoted by $\hat{D}_\Scal$, its projection in $\ol{\Mod}_{0,n}$ is denoted by $D_\Scal$.
\end{Definition}

We start by
\begin{Lemma}\label{lm:irrd:comp:boundary:is:div}
Every irreducible component of $\partial \blowupsp$ is a divisor
\end{Lemma}
\begin{proof}
Since $\partial \ol{\Mod}_{0,n}$ is locally defined in $\ol{\Mod}_{0,n}$ by a single equation, so is $\partial \blowupsp$.
\end{proof}

The first family of irreducible components of $\partial \blowupsp$ arise from the irreducible components of $\partial\ol{\Mod}_{0,n}$.
Consider a boundary divisor $D_\Scal \subset \partial\ol{\Mod}_{0,n}$ for some $\Scal=\{I_0,I_1\} \in \Pcal$. By Proposition~\ref{prop:supp:I:princ:comp}, we know that $D_\Scal$ is not contained in the support of $\Ical$.
Let $\hat{D}_\Scal$  denote the {\em proper transform} of $D_\Scal$ in $\blowupsp$.
Then $\hat{D}_\Scal$ is an irreducible component of $\partial \blowupsp$.
Note that $\hat{D}_\Scal$ is isomorphic to the blow-up of $D_\Scal$ along the ideal sheaf $i_\Scal^{-1}\Ical$, where $i_\Scal: D_\Scal \to \ol{\Mod}_{0,n}$ is the natural embedding (see\cite[Chap.II, Cor.7.15]{Hart}).

We now look for irreducible components $\hat{D}$ of $\partial \blowupsp$ whose projection in $\ol{\Mod}_{0,n}$ is contained in the support of $\Ical$.
Recall that a stratum of $\ol{\Mod}_{0,n}$ is the set of points that parametrize pointed stable curves which have the same dual graph.
Each stratum of $\ol{\Mod}_{0,n}$ is a (connected) submanifold of $\ol{\Mod}_{0,n}$ whose codimension is equal to the number of edges in the dual graph.

\begin{Lemma}\label{lm:str:proj:div:blowup}
Let $S$ be a stratum of $\partial \Mod_{0,n}$. If $S\cap \hat{p}(\hat{D})\neq \varnothing$, then $\ol{S}\subset \hat{p}(\hat{D})$.
\end{Lemma}
\begin{proof}
Let $r$ be the codimension of $S$.
Consider a point $\xx \in S \cap \hat{p}(\hat{D})$.
Let $\Ucal_\xx$ be a  neighborhood of $\xx$ as in \textsection\ref{subsec:blowup:construct}, and $\widehat{\Ucal}_\xx=\hat{p}^{-1}(\Ucal_\xx)$. Then $\hat{D}\cap \widehat{\Ucal}_\xx$ is an irreducible component of the algebraic set defined by the equations
$$
v_it^{\beta_j}=v_jt^{\beta_i}, \; i,j=0,\dots,r, \text{ and } \prod_{\alpha\in \NN(\xx)}t_\alpha=0,
$$
in $\Ucal_\xx\times \CP^r$.
By definition, $S$ is defined by the equations $t_\alpha=0$, for all $\alpha \in \NN(\xx)$.
Remark that  $\Ucal_\xx\simeq \Delta^{\NN(\xx)}\times U$, and the equations defining $\hat{D}$ do not involve the parameters on $U$. Therefore, if $\hat{D}$ contains a point $(0,u,[v]) \in \Delta^{\NN(\xx)}\times U\times\Pb^{r}$, then $\{0\}\times U\times \{[v]\} \subset \hat{D}$, which means that $\hat{p}(\hat{D})$ contains an open subset of $S$. This implies that  $S\subset \hat{p}(\hat{D})$ (because $\hat{p}(\hat{D})\cap S$ must be a Zariski closed subset of $S$).
Since $\hat{p}(\hat{D})$ is a closed subset of $\ol{\Mod}_{0,n}$, we conclude that $\ol{S}\subset \hat{p}(\hat{D})$.
\end{proof}
\begin{Corollary}\label{cor:str:proj:div:blowup}
There is a stratum $S$ of $\ol{\Mod}_{0,n}$ such that  $\hat{p}(\hat{D})=\ol{S}$.
\end{Corollary}
\begin{proof}
We have a partial order $\leqslant$  on the set of strata of $\ol{\Mod}_{0,n}$ defined by $S' \leqslant S$ if and only if $S'\subset \ol{S}$.
Note that if $S'\nsubseteq \ol{S}$ and $S \nsubseteq \ol{S'}$, then $\ol{S'}\cup \ol{S}$ is reducible.
By Lemma~\ref{lm:str:proj:div:blowup}, $\hat{p}(\hat{D})$ is a union of strata of $\ol{\Mod}_{0,n}$. We claim that there is a unique maximal stratum in $\hat{p}(\hat{D})$. Indeed, if $\hat{p}(\hat{D})$ contains two maximal strata, say $S_1$ and $S_2$, then  $\ol{S}_1$ and $\ol{S}_2$ are two distinct irreducible components of $\hat{p}(\hat{D})$, which is impossible since $\hat{p}(\hat{D})$ is irreducible.

Let $S$ be the unique maximal stratum in $\hat{p}(\hat{D})$. Then we have $\hat{p}(\hat{D})=\ol{S}$.
\end{proof}

\begin{Proposition}\label{prop:proj:div:part:codim:2}
Let $S$ be a stratum of $\ol{\Mod}_{0,n}$ of codimension at least $2$.
Let  $\{I_0,\dots,I_r\}$ be the partition of $\{1,\dots,n\}$ and $\T_S$  the dual graph associated to $S$.
The vertices of $\T_S$ are denoted by $v_0,\dots,v_r$, where $v_j$ is the vertex associated to $I_j$.
Assume that $\ol{S}=\hat{p}(\hat{D})$ for some irreducible component $\hat{D}$ of $\partial \blowupsp$.
Then up to renumbering of the subsets in $\{I_0,\dots,I_r\}$, we have
\begin{equation}\label{eq:part:bdry:div:codim:2}
\left\{\begin{array}{ll}
\sum_{i \in I_0}\mu_i < 1,&  \\
\sum_{i \in I_j}\mu_i > 1, & \text{ for all } j=1,\dots,r,
\end{array}
\right.
\end{equation}
and for every $j\in \{1,\dots,r\}$ there is an edge in $\T_S$ joining  $v_0$ to  $v_j$. Moreover, we have $\hat{D}=\ol{\hat{p}^{-1}(S)}$ and for all $\xx \in S, \; \hat{p}^{-1}(\xx)\simeq \CP^{r-1}$.

Conversely, assume that $\{I_0,\dots,I_r\}$ satisfies \eqref{eq:part:bdry:div:codim:2}, and every  edge of $\T_S$ joins $v_0$ to a vertex in $\{v_1,\dots,v_r\}$. Then $\ol{\hat{p}^{-1}(S)}$ is an irreducible component of $\partial \blowupsp$.
\end{Proposition}
\begin{proof}
Let $\xx$ be a point in $S$. Let $r_0$ be the number of principal subcurves  of $C_\xx$.
Since $C_\xx$ must have at least one non-principal component, we have $r_0 \leq r$.

By Corollary~\ref{cor:generators:I}, the ideal $\Ical$ is generated by $r_0$ functions $(t^{\beta^*_1},\dots,t^{\beta^*_{r_0}})$ in a neighborhood of $\xx$.
Thus by construction we have $\hat{p}^{-1}(\Ucal_\xx) \subset \Ucal_\xx\times \CP^{r_0-1}$, where $\Ucal_\xx$ is a neighborhood of $\xx$ as in \textsection~\ref{subsec:blowup:construct}.
In particular, we have $\hat{p}^{-1}(S)\subset S \times \CP^{r_0-1}$.
Since $\hat{D}\subset \hat{p}^{-1}(\ol{S})$ by Corollary~\ref{cor:str:proj:div:blowup}, we have that
$$
\dim{\hat{D}}=n-4  \leq \dim S+\dim \CP^{r_0-1},
$$
which implies
$$
n-4 \leq \dim S+ r_0-1=n-4-r+r_0.
$$
The condition $r_0 \leq r$ then implies that $r=r_0$, which means that $C_\xx$ has $r$ principal components and a unique non-principal component.
As usual, denote the irreducible components of $C_\xx$ by $C^0_\xx,\dots,C^r_\xx$ , where $C^j_\xx$ is the component corresponding to the vertex $v_j$ of $\T_S$, and $C^0_\xx$ is the unique non-principal component.
Since two principal subcurves cannot be adjacent, there must be a node between $C^0_\xx$ and $C^j_\xx$, for all $j\geq 1$.
Therefore the partition $\{I_0,\dots,I_r\}$ satisfies \eqref{eq:part:bdry:div:codim:2}.
Since in this case $\dim \hat{D}=\dim S+r-1$, the fibers $\hat{p}^{-1}(\xx)$ must be $\CP^{r-1}$ for all  $\xx \in S$. It follows that $\hat{D}=\ol{\hat{p}^{-1}(S)}$.

\medskip

For the converse, suppose that the partition $\{I_0,\dots,I_r\}$ satisfies \eqref{eq:part:bdry:div:codim:2}. For $j=1,\dots,r$, set $m_j=d(\mu(I_j)-1) \in \Z_{>0}$. Let $t_j,\; j=1,\dots,r$, be the coordinate function associated with the node between $C^0_\xx$ and  $C^j_\xx$. Then by definition
\begin{itemize}
\item $t^{\beta_0}=\prod_{j=1}^r t_j^{m_j}$,

\item $t^{\beta_k}=\left(\prod_{j=1}^r t_j^{m_j}\right)/t_k^{m_k}, \; k=1,\dots,r$.
\end{itemize}
Since $\Ical_{\Ucal_\xx}$ is generated by $\langle t^{\beta_1},\dots,t^{\beta_r}\rangle$, we have
$$
\widehat{\Ucal}_\xx=\hat{p}^{-1}(\Ucal_\xx)=\{(t,u,[v_1:\dots:v_r])\in \Ucal_\xx\times\CP^{r-1}, \; v_jt^{\beta_k}=v_kt^{\beta_j}, \; j,k=1,\dots,r\}.
$$
Note that the equation $v_jt^{\beta_k}=v_kt^{\beta_j}$ is actually equivalent to $v_jt_j^{m_j}=v_kt_k^{m_k}$.
It is straightforward to check that $\hat{p}^{-1}(\xx)\simeq \CP^{r-1}$ for all $\xx \in S$.
Hence $\hat{p}^{-1}(S)$ is a $\CP^{r-1}$-bundle over $S$. In particular $\dim \hat{p}^{-1}(S)=n-4$ and  $\ol{\hat{p}^{-1}(S)}$ is an irreducible component of $\partial \blowupsp$.
\end{proof}
\begin{Remark}\label{rk:fiber:of:blowup}\hfill
\begin{itemize}
\item[$\bullet$] In general, we do not have $\hat{D}=\hat{p}^{-1}(\ol{S})$.

\item[$\bullet$] It can be shown that for all $\xx \in \ol{\Mod}_{0,n}$, the fiber $\hat{p}^{-1}(\xx)$ is the projective space $\CP^{r_0-1}$, where $r_0$ is the number of principal subcurves of $C_\xx$.
\end{itemize}
\end{Remark}

\subsection*{Proof of Theorem~\ref{th:part:bdry:div:blowup}}
\begin{proof}
By Lemma~\ref{lm:irrd:comp:boundary:is:div}, every irreducible component of $\partial \blowupsp$ is a divisor.
Conider  an irreducible component $\hat{D}$ of $\partial\blowupsp$.
By Proposition~\ref{prop:proj:div:part:codim:2}, there is a  the stratum  $S$ of $\partial \ol{\Mod}_{0,n}$ such that $\hat{D}=\ol{\hat{p}^{-1}(S)}$. Let $r$ be the codimension of $S$ in $\ol{\Mod}_{0,n}$. If $r=1$, then $S$ is the set of generic points in a boundary divisor in $\partial \ol{\Mod}_{0,n}$, which corresponds to a partion $\{I_0,I_1\} \in \Pcal$. In this case $\hat{D}$ is the proper transform of the corresponding boundary divisor of $\ol{\Mod}_{0,n}$. It is clear that the correspondence $\hat{D} \mapsto \{I_0,I_1\}$ is one-to-one.

If $r>1$, then it follows from Proposition~\ref{prop:proj:div:part:codim:2} that
the partition $\{I_0,I_1,\dots,I_r\}$ associated to $S$ satisfies the required condition.
The correspondence $\hat{D} \mapsto \{I_0,\dots,I_r\}$ is also one-to-one because given $\{I_0,\dots,I_r\}$, there is a unique configuration of the dual graph such that the curves parametrized by $S$ has exactly $r$ principle components.
\end{proof}


\section{Geometric interpretation}\label{sec:geom:interpret}
\subsection{Differentials associated to points in the blow-up}\label{subsec:ddiff:assoc:blowup}
We now give the geometric interpretation of the points in $\blowupsp$.
Let $\xx$ be a point in  $\ol{\Mod}_{0,n}$, and $\Ucal_\xx$ a neighborhood of $\xx$ satisfying the conditions (i) and (ii) in \textsection\ref{subsec:blowup:construct}.
In what follows, to simplify the notation, we will write $\Ucal$ and $\widehat{\Ucal}$ instead of $\Ucal_\xx$ and $\widehat{\Ucal}_\xx$.
Recall that from Proposition~\ref{prop:Kaw:trivial:ratio} we have $(r+1)$ holomorphic sections $t^{\beta_0}\Phi_0,\dots,t^{\beta_r}\Phi_r$ of the line bundle $\Kc^{(d)}_{0,n}$ over $\ol{\Cc}_{0,n|\Ucal}$.
By multiplying $\Phi_1,\dots,\Phi_{r}$ by some non-vanishing holomorphic functions on $\Ucal$, we can actually write
\begin{equation}\label{eq:Kaw:sect:equal}
t^{\beta_0}\Phi_0=\dots=t^{\beta_{r-1}}\Phi_{r-1}=t^{\beta_r}\Phi_r.
\end{equation}
For every $\xx'\in \Ucal$, denote by $\Phi_j(\xx')$ the restriction of $\Phi_j$ to the curve $C_{\xx'}$.
Given $v\in \C^{r+1}\setminus\{0\}$, we denote by $[v]\in \CP^r$ the line generated by $v$.
Consider a vector $v=(v_0,\dots,v_r)\in \C^{r+1}\setminus\{0\}$ such that $(\xx', [v]) \in \hat{p}^{-1}(\{\xx'\})$.
Assume that $\xx'\in \Ucal^*:=\Ucal\cap\Mod_{0,n}$ (that is $t_\alpha(\xx')\neq 0$ for all $\alpha \in \NN(\xx)$).
In this case, we have
$$
\hat{p}^{-1}(\{\xx'\})=(\xx',[t^{\beta_0}(\xx'):\dots:t^{\beta_r}(\xx')]) \in \Ucal\times\CP^{r}.
$$
Hence $v=\lambda\cdot(t^{\beta_0}(\xx'),\dots,t^{\beta_r}(\xx'))$ for some $\lambda\in \C^*$.
Therefore, $v_0\Phi_0(\xx')=\dots=v_r\Phi_r(\xx')$, which means that  all the elements in the family $\{v_0\Phi_0(\xx'),\dots,v_r\Phi_r(\xx')\}$ give the same $d$-differential  on $C_{\xx'}$.
Thus to every element $(\xx', [v])\in \hat{p}^{-1}(\Ucal^*)$ we have an associated  $d$-differential on $C_{\xx'}$ determined up to multiplication by a scalar in $\C^*$, that is an element of $\Pb\stratesp$.

\medskip

More generally, for all $\xx'\in \Ucal$, let $C^0_{\xx'},\dots,C^{r'}_{\xx'}$ denote the irreducible components of $C_{\xx'}$.
Associated to $C_{\xx'}$ we have a partition $\{J_0,\dots,J_{r'}\}$ of the set $\{0,\dots,r\}$ defined as follows: let $f: C_{\xx'} \ra C_\xx$ be a (surjective) map that pinches some simple closed curves in $C_{\xx'}$ into nodes in $C_\xx$. Then for all $k\in\{0,\dots,r'\}$, $J_k$ is the set of indexes $i\in \{0,\dots,r\}$ such that $f(C^k_{\xx'})$ contains $C^i_\xx$.
Note that the partition $\{J_0,\dots,J_{r'}\}$ depends only on the stratum of $\xx'$ in $\ol{\Mod}_{0,n}$.

\medskip

For all $i,j \in \{0,\dots,r\}$, by definition, we must have $v_it^{\beta_j}=v_jt^{\beta_i}$ and $t^{\beta_i}\Phi_i=t^{\beta_j}\Phi_j$. But since $t^{\beta_i}(\xx')$ and $t^{\beta_j}(\xx')$ can vanish, $v_i\Phi_i(\xx')$ may not be equal to $v_j\Phi_j(\xx')$ in general. However, we have
\begin{Lemma}\label{lm:differentials:on:cpnts}
If $i,j\in J_k$ for some $k\in \{0,\dots,r'\}$, then $v_i\Phi_i(\xx')$ and $v_j\Phi_j(\xx')$ give the same $d$-differential on the component $C^k_{\xx'}$.
\end{Lemma}
\begin{proof}
We first remark that the function $t^{\beta_i-\beta_j}:=t^{\beta_i}/t^{\beta_j}$ involves only $t_\alpha$ where the node $\alpha$ separates $C^i_\xx$ and $C^j_\xx$.
In $C_{\xx'}$ these nodes are smoothened, which means that $t_\alpha(\xx')\neq 0$ for all $\alpha$'s that separate $C^i_\xx$ and $C^j_\xx$.
Therefore, $t^{\beta_i-\beta_j}(\xx')\neq 0$. In particular, $t^{\beta_i-\beta_j}$ is a well defined holomorphic function which does not vanish in a neighborhood of $\xx'$.

Obviously, we only need to consider the case $(v_i,v_j)\neq (0,0)$. Without loss of generality, we can suppose that $v_j\neq 0$. By definition, there is a sequence $\{(\xx'_s,[v^{(s)}])\}_{s\in\N} \subset \hat{p}^{-1}(\Ucal^*)$, where $v^{(s)}=(v^{(s)}_0,\dots,v^{(s)}_r)$, converging to $(\xx',[v])$.
Since $\xx'_s\in \Ucal^*$, we have $v^{(s)}_i/v^{(s)}_j=t^{\beta_i}(\xx'_s)/t^{\beta_j}(\xx'_s)=t^{\beta_i-\beta_j}(\xx'_s)$.
It follows that
$$
\frac{v_i}{v_j}=\lim_{s\to\infty}\frac{v^{(s)}_i}{v^{(s)}_j}=\lim_{s\to \infty}t^{\beta_i}(\xx'_s)/t^{\beta_j}(\xx'_s) =t^{\beta_i-\beta_j}(\xx').
$$
Thus we have $\Phi_j(\xx')/\Phi_i(\xx')=t^{\beta_i-\beta_j}(\xx')=\frac{v_i}{v_j}$, which implies that  $v_i\Phi_i(\xx')=v_j\Phi_j(\xx')$.
We now observe that  the restrictions of $\Phi_i(\xx')$ and $\Phi_j(\xx')$  to the component $C^k_{\xx'}$ give non-zero $d$-differentials.
This is  because the restrictions of  the projections $\pi_i$ and $\pi_j$ to $C^k_{\xx'}$ give two isomorphisms from $C^k_{\xx'}$ onto $\CP^1$ (see Lemma~\ref{lm:univ:curv:geom:codim2}), and  $\Phi_i$ and $\Phi_j$ are the pullbacks   of  non-vanishing $d$-differentials on $\CP^1$ by $\pi_i$ and $\pi_j$ respectively.
Therefore, $v_i\Phi_i(\xx')$ and $v_j\Phi_j(\xx')$ give the same $d$-differential on $C^k_{\xx'}$.
\end{proof}

Denote by $q_k(\xx',v)$ the $d$-differential given by $v_j\Phi_j(\xx')$ on $C^k_{\xx'}$ for some $j\in J_k$.
By Lemma~\ref{lm:differentials:on:cpnts} $q_k(\xx',v)$ is independent of the choice of $j\in J_k$.
Thus for every $(\xx',[v]) \in \widehat{\Ucal}$, we can associate to $(\xx',[v])$ the tuple $(q_0(\xx',v),\dots,q_{r'}(\xx',v))$ of $d$-differentials on the components of $C_{\xx'}$, which are determined up to simultaneous multiplication by the same scalar in $\C^*$.

\subsection{Embedding into $\Pb\ol{\Hcal}^{(d)}_{0,n}$}\label{sec:proj:mer:ddiff}
Recall that we have a natural section $\ff: \Mod_{0,n} \ra \Pb\Hcal^{(d)}_{0,n}$ which associates to $\xx \sim  (\CP^1,x_1,\dots,x_n)$ the complex line generated by the $d$-differential $\prod_{i=1}^n(z-x_i)^{k_i}(dz)^d$ viewed as element of $H^0(C_\xx, \Kcal^{(d)}_{0,n|C_{\xx}})$.
By definition,  $\ff(\Mod_{0,n})=\Pb\stratesp$.
From what we have seen, the correspondence $(\xx',[v]) \mapsto \C^*\cdot(q_0(\xx',v),\dots,q_{r'}(\xx',v))$ provides us with a map
$\hat{\ff}_{\widehat{\Ucal}}: \widehat{\Ucal} \ra \Pb\ol{\Hcal}^{(d)}_{0,n|\Ucal}$.
We now show
\begin{Theorem}\label{th:inject:blowup:proj:bdl}
The maps $\hat{\ff}_{\widehat{\Ucal}}$ are proper injective morphisms of algebraic varieties and patch together to give an injective morphism $\hat{\ff}:\blowupsp \to \Pb\ol{\Hcal}^{(d)}_{0,n}$ extending the section $\ff: \Mod_{0,n} \to \Pb\Hcal^{(d)}_{0,n}$.
\end{Theorem}
\begin{proof}
We will define the morphism $\hat{\ff}$ via a section of the projective bundle $\hat{p}^*\Pb\ol{\Hc}^{(d)}_{0,n}$ over $\blowupsp$. To this purpose, we consider the variety $\widehat{\Ccal}_{\widehat{\Ucal}}:=\widehat{\Ucal}\times_{\Ucal}\ol{\Ccal}_{\Ucal}$, which is the pullback of the family $\ol{\Ccal}_{\Ucal}$ to $\widehat{\Ucal}$.
By the construction of \textsection\ref{subsec:univ:curv:near:bdry}, $\widehat{\Ccal}_{\widehat{\Uc}}$ is the subvariety of $(\CP^1)^{r+1}\times\widehat{\Uc}$ which is defined by the equations \eqref{eq:mult:nodes:fam:curv}.
We have the following commutative diagram
\begin{center}
 \begin{tikzpicture}[scale=0.3]
    \node (A) at (0,4) {$\widehat{\Ccal}_{\widehat{\Ucal}}$};
    \node (B) at (6,4) {$\ol{\Ccal}_{\Ucal}$};
    \node (C) at (0,0) {$\widehat{\Ucal}$};
    \node (D) at (6,0) {$\Ucal$};

    \path[->, font=\scriptsize, >= angle 60] (A) edge node[above]{$p_2$} (B);
    \path[->, font=\scriptsize, >= angle 60]
    (A) edge node[left]{$p_1$} (C)
    (C) edge node[above]{$\hat{p}$} (D)
    (B) edge node[right]{$p$} (D);
 \end{tikzpicture}
\end{center}
Since $\Kcal^{(d)}_{0,n|C_{\xx'}}\sim \Kcal^{(d)}_{\xx'}$ for all $\xx'\in \Ucal$, and $\dim H^1(C_{\xx'},\Kc^{(d)}_{\xx'})=0$ (cf. \eqref{eq:R1:vanish}), by Cohomology and Base Change Theorem (see \cite[Chap.III, Th.12.11]{Hart} or \cite[Chap. 9, Prop.3.3]{ACG2011}) we have
$$
p_{1*}p^*_2\Kcal^{(d)}_{0,n}\sim \hat{p}^*p_*\Kcal^{(d)}_{0,n}\sim \hat{p}^*\ol{\Hcal}^{(d)}_{0,n}
$$
which means that the sections of $\hat{p}^*\ol{\Hc}^{(d)}_{0,n}$ on $\widehat{\Ucal}$ correspond to sections of $p^*_2\Kcal^{(d)}_{0,n}$ on $\widehat{\Ccal}_{\widehat{\Ucal}}$.

For $i=0,\dots,r$, set
$$
\widehat{\Ucal}^i:=\{(t,u,[v])\in \widehat{\Ucal},  v_i\neq 0\}, \; \text{ and } \widehat{\Ccal}^i_{\widehat{\Ucal}}:=p^{-1}_1(\widehat{\Ucal}^i).
$$
For all $j\in \{0,\dots,r\}$, on $\widehat{C}_{\widehat{\Uc}}$ we have
$1=(t^{\beta_j}\Phi_j)/(t^{\beta_i}\Phi_i)=(v_j/v_i)\cdot (\Phi_j/\Phi_i)$,
or equivalently $\Phi_i=(v_j/v_i)\cdot\Phi_j$.
Note that $v_j/v_i$ is a regular function on $\widehat{\Uc}^i$. We claim that $\Phi_i$ defines a holomorphic section of the line bundle $p^*_2\Kcal^{(d)}_{0,n}$ over $\widehat{\Cc}^i_{\widehat{\Uc}}$.
Indeed, in the proof of Proposition~\ref{prop:Kaw:trivial:ratio} we actually showed that for every point $(z^0,t^0,u^0) \in \ol{\Cc}_\Uc$, there is some $j \in \{0,\dots,r\}$ such that $\Phi_j$ is a holomorphic section of $\Kc^{(d)}$ in a neighborhood of $(z^0,t^0,u^0)$. This implies that $\Phi_j$ gives a holomorphic section of $p_2^*\Kc^{(d)}$ in a neighborhood of $(z^0,t^0,u^0,[v^0])$ for all $(t^0,u^0, [v^0]) \in \hat{p}^{-1}(t^0,u^0)$.
If $(z^0,t^0,u^0, [v^0]) \in \widehat{\Cc}^i_{\widehat{\Uc}}$ then $\Phi_i=\frac{v_j}{v_i}\Phi_j$ is also a holomorphic section of $p^*_2\Kc^{(d)}_{0,n}$ near $(z^0,t^0,u^0,[v^0])$ since $v_j/v_i$ is a regular function on $\widehat{\Uc}^i$. The claim is then proved.

By definition, $\Phi_i$ gives a  holomorphic section of $p_{1*}p_2^*\Kcal^{(d)}_{0,n}\sim \hat{p}^*\Hc^{(d)}_{0,n}$ over $\widehat{\Ucal}^i$.
This section assigns to the point $(\xx',[v]) \in \widehat{\Ucal}^i$ the $d$-differential $\Phi_i(\xx')$ on $C_{\xx'}$.
We now remark that $\Phi_i(\xx')$ does not vanish identically on $C_{\xx'}$, which means that $\Phi_i(\xx')\neq 0 \in \ol{\Hc}^{(d)}_{0,n,\xx'}$.
Thus by taking the image in the associated projective bundle, we get a section $\hat{\Psi}_i$ of $\hat{p}^*\Pb\ol{\Hcal}^{(d)}_{0,n}$ over $\widehat{\Ucal}^i$.
For all $i,j \in \{0,\dots,r\}$, the sections $\hat{\Psi}_i$ and $\hat{\Psi}_j$ coincide in the intersection $\widehat{\Ucal}^i \cap \widehat{\Ucal}^j$.
Thus we obtain a map $\hat{\Psi}: \widehat{\Ucal} \ra \hat{p}^*\Pb{\Hcal}^{(d)}_{0,n|\Ucal}$.
By definition the map $\hat{\ff}_{\widehat{\Ucal}}$ is the composition $\hat{q}\circ\hat{\Psi}$, where $\hat{q}: \hat{p}^*\Pb\ol{\Hcal}^{(d)}_{0,n|\Ucal} \ra \Pb\ol{\Hcal}^{(d)}_{0,n|\Ucal}$ is the natural projection.
It follows that $\hat{\ff}_{\widehat{\Ucal}}$ is a morphism of algebraic varieties.

It is straightforward to check that $\hat{\ff}_{\widehat{\Ucal}}$ is injective.
Since $\blowupsp$ is the blow-up along the sheaf of ideals $\Ical$, the maps $\hat{\ff}_{\widehat{\Ucal}}$ patch together to give an embedding of $\blowupsp$ into $\Pb\ol{\Hcal}^{(d)}_{0,n}$ whose restriction to $\Mod_{0,n}$  equals  $\ff$.
\end{proof}

\subsection{Proof of Theorem~\ref{th:IVC:is:blowup}}\label{sec:prf:th:IVC:is:blowup}
\begin{proof}
The first assertion of Theorem~\ref{th:IVC:is:blowup} is the content of Theorem~\ref{th:ideal:sheaf:def}. That $\clpstratesp$ is isomorphic to $\blowupsp$ follows readily from Theorem~\ref{th:inject:blowup:proj:bdl}.  Specifically, since $\Mod_{0,n}$ is dense in $\blowupsp$ and $\blowupsp$ is compact, $\hat{\ff}(\blowupsp)$ equals the closure of $\hat{\ff}(\Mod_{0,n})=\Pb\stratesp$ in $\Pb\ol{\Hcal}^{(d)}_{0,n}$, that is $\hat{\ff}(\blowupsp)=\clpstratesp$.
Since $\hat{\ff}$ is an embedding, we conclude that $\blowupsp$ is isomorphic to $\clpstratesp$.
The last assertion has been proved in Theorem~\ref{th:part:bdry:div:blowup}.
\end{proof}
\section{The exceptional divisor of the blow-up}\label{sec:except:div}
Let $\xx$ be a point in $\ol{\Mod}_{0,n}$.
Let $\Uc_\xx\simeq \Delta^{\NN(\xx)}\times U$ and $(t_\alpha)_{\alpha \in \NN(\xx)}$ be as in \textsection \ref{subsec:univ:curv:near:bdry}.
%
%
Recall that $\widehat{\Uc}_\xx$ is covered by the open subsets
$$
\widehat{\Uc}_\xx^k:=\{(t,u,[v_0:\dots:v_r]) \in \widehat{\Uc}_\xx, \; v_k\neq 0\}, \qquad k=0,\dots,r.
$$
By definition,  $\hat{p}^{-1}\cdot\Ic_{\widehat{\Uc}_\xx}$ is  generated by $\langle t^{\beta_1},\dots,t^{\beta_r}\rangle$ on $\widehat{\Ucal}_\xx$.
In  $\widehat{\Uc}_\xx^k$,  we can write $t^{\beta_j}=\frac{v_j}{v_k}t^{\beta_k}$, which means that
$\hat{p}^{-1}\cdot\Ic_{\widehat{\Uc}^k_\xx}$ is generated by $t^{\beta_k}$. For $j\neq k$, we have $t^{\beta_k}/t^{\beta_j}=v_k/v_j$ that is an invertible regular function on $\widehat{\Uc}_\xx^j\cap \widehat{\Uc}_\xx^k$. Thus the data $(\widehat{\Uc}^k_\xx, t^{\beta_k})$ define a Cartier divisor $\Ec$ on $\blowupsp$ which satisfies $\Ec\sim \hat{p}^{-1}\Ic$.  This Cartier divisor is called the exceptional divisor of $\blowupsp$.

Our goal now is to determine the Weil divisor  associated to $\Ec$. To this purpose, we need to compute the vanishing order of the functions defining $\Ec$ along the components of $\partial\blowupsp$.
Consider a partition $\Sc=\{I_0,\dots,I_r\}\in \hat{\Pc}(\mu)$.
Let $\xx$ be a point in the stratum $D^*_\Sc\subset \ol{\Mod}_{0,n}$ associated to $\Sc$.
Let $t_j, \; j=1,\dots,r$, be the  coordinate function on $\Uc_\xx$ which corresponds to the node between $C^j_\xx$ and $C^0_\xx$.
Set $m_j:=d(\mu(I_j)-1) \in \Z_{\geq 0}$, and define
\begin{equation}\label{def:mult:bdry:div:P1}
m(\Sc):=\prod_{j=1}^rm_j=d^r\cdot \prod_{j=1}^r(\mu(I_j)-1).
\end{equation}
Note that by definition, we have
\begin{itemize}
\item $t^{\beta_0}=\prod_{j=1}^rt_j^{m_j}$,

\item $t^{\beta_k}=\frac{\prod_{j=1}^rt^{m_j}}{t_k^{m_k}}$, for $k=1,\dots,r$.
\end{itemize}
Since $t^{\beta_0}=t_1^{m_1}t^{\beta_1}$, we get that $\Ic_{\Uc}=\langle t^{\beta_1},\dots,t^{\beta_r}\rangle$. Hence $\widehat{\Uc}_\xx$ is isomorphic to the irreducible component of
$$
\{(t,u,[v_1:\dots:v_r]) \in \Uc_\xx\times\Pb^{r-1}_\C, \; v_jt^{\beta_k}=v_kt^{\beta_j}, \; \forall j,k\in\{1,\dots,r\}\} \subset \Uc_\xx\times\Pb^{r-1}_\C
$$
which contains the subset $\{(t,u,[t^{\beta_1}:\dots:t^{\beta_r}]), \; (t,u)\in (\Delta^*)^r\times U\}$.
Note that in $\widehat{\Uc}_\xx$, $\hat{D}_\Sc$ is defined by the equations $t_1=\dots=t_r=0$.

\begin{Lemma}\label{lm:vanish:ord:bdry:div:1}
For all $j\in\{1,\dots,r\}$, the vanishing order of $t_j$ along $\hat{D}_\Sc$ is equal to $\frac{\prod_{i=1}^r m_i}{m_j}$.
\end{Lemma}
\begin{proof}
Over $\Uc^*_\xx\times\CP^{r-1}$, since $t_1\cdots t_r \neq 0$, the defining equations  of $\widehat{\Uc}_\xx$ are equivalent to
\begin{equation}\label{eq:def:eqns:div:P1}
v_1t_1^{m_1}=\dots=v_rt_r^{m_r}.
\end{equation}
Therefore we can actually use \eqref{eq:def:eqns:div:P1} to define $\widehat{\Uc}_\xx$.
For concreteness, let us assume that $j=r$.
Consider the open affine
$$\widehat{\Uc}^r_\xx:=\{(t,u,[v_1:\dots:v_r]) \in \widehat{\Uc}_\xx, \; v_r\neq 0\}
$$
of $\widehat{\Uc}_\xx$. Note that $\widehat{\Uc}_\xx^r$ is defined in $\Uc_\xx\times\C^{r-1}$ by the equations
$$
t_r^{m_r}=v_jt^{m_j}, \; \text{ for all } j \in \{1,\dots,r-1\}.
$$
In $\widehat{\Uc}^r_\xx$, the divisor $\hat{D}_\Sc$ is defined by the equations $t_1=\dots=t_r=0$.
Let $A$ be the localization of the coordinate ring of $\widehat{\Uc}^r_\xx$ at the ideal generated by $\{t_1,\dots,t_r\}$.
By definition, the vanishing order of $t_r$ along $\hat{D}_\Sc$ is given by $\ell(A/(t_r))$, that is the length of $A/(t_r)$ as an $A$-module (see \cite[\textsection 1.2]{Ful}).

Remark that $t_j^{m_j}=\frac{t_r^{m_r}}{v_j}$ in $A$. Therefore, every element of $A$ is represented by a quotient $P/Q$, with $P,Q \in \C(v_1,\dots,v_{r-1})[t_1,\dots,t_r]$, where for $j=1,\dots,r-1$, the degree of $t_j$ in $P$ and $Q$ is at most $m_j-1$, and the $(t_1,\dots,t_r)$-free term in $Q$ is non-zero.
We have
$$
A/(t_r) \simeq \C(v_1,\dots,v_{r-1})[t_1,\dots,t_{r-1}]/(t_1^{m_1},\dots,t_{r-1}^{m_{r-1}}).
$$
Observe that $\C(v_1,\dots,v_{r-1})$ is the residue field of $A$, and we have
$$
\dim_{\C(v_1,\dots,v_{r-1})}\C(v_1,\dots,v_{r-1})[t_1,\dots,t_{r-1}]/(t_1^{m_1},\dots,t_{r-1}^{m_{r-1}})=\prod_{j=1}^{r-1}m_j.
$$
Therefore, $\ell(A/(t_r))=\prod_{j=1}^{r-1}m_j$ (see \cite[Ex. A.1.1]{Ful}) as desired.
\end{proof}
 As a consequence we get

\begin{Theorem}\label{th:coeff:bound:div:in:E}
Let $\Sc=\{I_0,I_1,\dots,I_r\}$ be a partition in $\hat{\Pc}(\mu)$.
Then the coefficient of $\hat{D}_\Sc$ in $\Ec$ is equal to $(r-1)\cdot m(\Sc)$. In particular, we have
\begin{equation*}
\Ec\sim \sum_{\Sc\in \hat{\Pc}(\mu)}(|\Sc|-2)\cdot m(\Sc)\cdot \hat{D}_\Sc.
\end{equation*}
\end{Theorem}
\begin{proof}
It is enough to consider an open affine that meets $\hat{D}_\Sc$.
We can take this open affine to be $\widehat{\Uc}^r_\xx \subset \widehat{\Uc}_\xx$. On  $\widehat{\Uc}^r_\xx$, the Cartier divisor $\Ec$  is defined by the function $t^{\beta_r}=\prod_{j=1}^{r-1}t_j^{m_j}$. By Lemma~\ref{lm:vanish:ord:bdry:div:1}, we know that the order of $t_j$ along $\hat{D}_\Sc$ is $\frac{\prod_{i=1}^{r}m_i}{m_j}$. Thus the order of $t_j^{m_j}$ is $\prod_{i=1}^r m_i = m(\Sc)$, and the order of $t^{\beta_r}$ is $(r-1)\cdot m(\Sc)$.
\end{proof}

\begin{Remark}\label{rk:coeff:Weil:div:0}
For all  $\Sc=\{I_0,I_1\}\in \Pc$, the coefficient of $\hat{D}_\Sc$ in $\Ec$ is zero, which can be explained by the fact that the divisor $D_\Sc \subset \partial\ol{\Mod}_{0,n}$ is not contained in the support of $\Ic$ (cf. Proposition~\ref{prop:supp:I:princ:comp}).
\end{Remark}
\section{The Kawamata line bundle}\label{sec:kaw:ln:bdl}
\subsection{Definition}\label{subsec:kaw:ln:bdl:def}
To prove Theorem~\ref{th:vol:n:inters:g0}, we will compare $\hat{\Lc}_\mu$ with (the pullback) of a line bundle $\bar{\Lc}_\mu$ over $\ol{\Mod}_{0,n}$. To define $\bar{\Lc}_\mu$, we consider the following line bundle over $\ol{\Cc}_{0,n}$: recall that for any $\Sc=\{I_0,I_1\}\in \Pcal$, $D_\Sc$ is the irreducible component of $\partial \ol{\Mod}_{0,n}$ associated to $\Sc$.
A generic point $\xx \in D_\Scal$ represents a stable curve $C_\xx$ with two irreducible components $C^0_\xx$ and $C^1_\xx$, where $C^j_\xx$ contains the $i$-th marked point if and only if $i \in I_j$.
By definition, $p^{-1}(D_\Scal)$ is a divisor of $\ol{\Ccal}_{0,n}$ with two irreducible components denoted by $F^0_{\Scal}$ and $F^1_{\Scal}$ such that $F^j_\Scal\cap C_\xx=C^j_\xx$. We then define
\[
\Kcal_\mu:=d\left( K_{\ol{\Ccal}_{0,n}/\ol{\Mod}_{0,n}}+\sum_{i=1}^n\mu_i\Gamma_i-\sum_{\Scal \in \Pcal}\mu_\Scal F^0_\Scal \right).
\]
(recall that $\mu_\Sc=\mu(I_1)-\mu(I_0)\geq 0$).
This bundle  was first introduced in \cite{Kawa}, we will call it the {\em Kawamata line bundle}. Note however that our convention on the numbering  $I_0,I_1$ is the converse of the one in \cite{Kawa}.

For any $\xx\in \ol{\Mod}_{0,n}$, denote by $\Kc_{\mu,\xx}$ the restriction of $\Kc_\mu$ to the curve $C_\xx \simeq p^{-1}(\{\xx\})$.
If $\xx \sim(\Pb^1,x_1,\dots,x_n) \in \Mod_{0,n}$, then $\dim H^0(C_\xx,\Kcal_{\mu,\xx})=1$ (since $\mathrm{deg}(\Kc_{\mu,\xx})=0$).
It follows that $p_*\Kcal_{\mu| \Ccal_{0,n}}$ gives a line bundle $\Lcal_\mu$ over $\Mod_{0,n}$.
The complement of the zero section in the total space of $\Lcal_\mu$ can be identified with $\stratesp$.
Our goal in this section is to show the following
\begin{Theorem}\label{th:push:Kaw:ln:bdl}
There is a line bundle $\bar{\Lcal}_\mu$ over $\ol{\Mod}_{0,n}$ such that $\Kcal_\mu\sim p^*\bar{\Lcal}_\mu$ and $\bar{\Lcal}_{\mu|\Mod_{0,n}} \sim \Lcal_\mu$. Moreover,  $\bar{\Lcal}_\mu$ is isomorphic to the line bundle associated to the divisor
\begin{align*}
\Dcal_\mu &= \frac{d}{(n-2)(n-1)}\sum_{\Scal=\{I_0,I_1\}\in \Pcal}(|I_0|-1)(|I_1|-1-(n-1)\mu_\Scal)\cdot D_\Scal\\
&\sim \frac{d}{2}\left(\sum_{i=1}^n-\mu_i\psi_i+\sum_{\Scal\in \Pcal}(1-\mu_\Scal)D_\Scal\right).
\end{align*}
\end{Theorem}

\begin{Remark}\label{rk:push:Kaw:ln:bdl}
In  the case all the weights $\mu_i$ are positive, Theorem~\ref{th:push:Kaw:ln:bdl} is a consequence of a result by Kawamata (see\cite[Th.4]{Kawa}).
\end{Remark}

\subsection{Trivializing sections of the Kawamata line bundle}\label{subsec:kaw:triv:sect}
Let $\xx$ be a point in a stratum of codimension $r$ in $\ol{\Mod}_{0,n}$. Let $\Uc$, $\ol{\Cc}_\Uc$ be as in \textsection\ref{subsec:univ:curv:near:bdry}.
We will prove
\begin{Proposition}\label{prop:kaw:trivial:sect}
Let $\{t^{\beta_j}, \, j=0,\dots,r\}$ be the monomials defined in \eqref{eq:def:t:beta} and $\{\Phi_j,\; j=0,\dots,r\}$ be as in \eqref{eq:def:Phi}. Then for all $j\in \{0,1,\dots,r\}$,  $t^{\beta_j}\Phi_j$  gives a trivializing section of the Kawamata line bundle $\Kcal_\mu$ on $\ol{\Ccal}_{\Ucal}$.
\end{Proposition}
\begin{proof}
Let $(z^0,t^0,u^0)$ be a point in $\ol{\Ccal}_{\Ucal}$. We will show that for some $j\in \{0,\dots,r\}$, $t^{\beta_j}\Phi_j$ gives a trivializing section of $\Kc_\mu$ on a neighborhood of $(z^0,t^0,u^0)$ and use Proposition~\ref{prop:Kaw:trivial:ratio} to conclude.
Assume first that $z^0$ is a smooth point of $C(t^0,u^0)$, which means that $z^0$ is contained in a unique irreducible component $C^k(t^0,u^0)$ of $C(t^0,u^0)$. By Lemma~\ref{lm:univ:curv:geom:codim2}, for some $j\in \{0,\dots,r\}$, the projection $\pi_j$ restricts to an isomorphism from $C^k(t^0,u^0)$ onto $\CP^1$. Thus $\Phi_j$ is a trivializing section of the line bundle $d(K_{\ol{\Ccal}_{\Ucal}/\Ucal}+\sum_{i=1}^n\mu_i\Gamma_i)$ in a neighborhood of $(z^0,t^0,u^0)$.
It remains to check that $t^{\beta_j}$ is a trivializing section of $\sum_{\Scal \in \Pcal} -d\mu_\Scal F^{0}_\Scal$ near $(z^0,t^0,u^0)$.
Since $\Ucal$ only  intersects  $D_\Scal$ if $\Scal$ corresponds to a node of $C_\xx$,  all we need to show is that $t^{\beta_j}$ is a trivializing section of the line bundle $\sum_{\alpha\in\NN(\xx)}-d\mu_{\Sc_\alpha}F^0_{\Sc_\alpha}$.   For simplicity, in what follows, for all $\alpha \in \NN(\xx)$ we will write  $\mu_\alpha$ and $F^j_\alpha$ ($j=0,1$) instead of $\mu_{\Sc_\alpha}$ and $F^j_{\Sc_\alpha}$.

Since $(t^0,u^0) \in \Ucal$, the set of partitions $\Sc\in \Pc$ such that $D_{\Sc}$ contains $(t^0,u^0)$ is in bijection with the set $Z(t^0)=\{\alpha\in \NN(\xx), \; t_\alpha=0\}$. Define $Z_0(t^0):=\{\alpha \in \NN(\xx), \; C^k(t^0,u^0)\subset F^0_{\alpha}\}$. Since $(z^0,t^0,u^0) \in F^0_\alpha$ if and only if $\alpha\in Z_0(t^0)$, a trivializing section of the divisor $\sum_{\alpha\in \NN(\xx)}-d\mu_{\alpha}F^0_{\alpha}$ in a neighborhood of $(z^0,t^0,u^0)$ is given by $\prod_{\alpha \in Z_0(t^0)}t_\alpha^{d\mu_{\alpha}}$. By definition, we immediately get
\begin{equation*}\label{eq:triv:sect:F0}
\prod_{\alpha \in Z_0(t^0)}t_\alpha^{d\mu_{\alpha}}=\prod_{\alpha \in Z_0(t^0)}t_\alpha^{\beta_{j,\alpha}}=\prod_{\alpha \in Z(t^0)}t_\alpha^{\beta_{j,\alpha}}.
\end{equation*}

Since $t^0_\alpha \neq 0$ for all $\alpha\in \NN(\xx)\setminus Z(t^0)$, we conclude that  $t^{\beta_j}=\prod_{\alpha\in \NN(\xx)}t_\alpha^{\beta_{j,\alpha}}$ is a trivializing section of $\sum_{\Scal \in \Pcal} -d\mu_\Scal F^{0}_\Scal$ near $(z^0,t^0,u^0)$.
As a consequence, $t^{\beta_j} \Phi_j$ is a trivializing section of the Kawamata line bundle in a neighborhood of $(z^0,t^0,u^0)$.

\medskip

Assume now that $z^0$ is a node of $C(t^0,u^0)$ which is the intersection of two components $C^{k}(t^0,u^0)$ and $C^{k'}(t^0,u^0)$.
Let $\alpha=\{j,j'\} \in Z(t^0)$ be the corresponding node in $C_\xx$, where  the restriction of $\pi_j$ (resp. of $\pi_{j'})$) to $C^k(t^0,u^0)$ (resp. to $C^{k'}(t^0,u^0)$) is an isomorphism onto $\CP^1$.
We can also assume that $\beta_{j,\alpha}=0$ and $\beta_{j',\alpha}=d\mu_\alpha$, which means that  $C^{k'}(t^0,u^0)$ is contained in the divisor $F^0_{\alpha}$. Note that near the point $(z^0,t^0,u^0)$, the divisor $F^0_{\alpha}$ is defined by the equation $z_j-b_{jj'}=0$.

Recall from \eqref{eq:phi:near:node} that we have
$$
\Phi_j=(-1)^{\sum_{i\in I_{0,\alpha}}k_i}(z_j-b_{jj'})^{d\mu_\alpha}\cdot \prod_{i\in I_{1,\alpha}}(z_j-a_{ji})^{k_i} \cdot \prod_{i\in I_{0,\alpha}}\left(\frac{z_{j'}-a_{j'i}}{a_{j'i}-b_{j'j}}\right)^{k_i} \cdot \left( \frac{dz_j}{z_j-b_{jj'}} \right)^d.
$$
where $z_j-a_{ji}$ with $i \in I_{1,\alpha}$, and $(z_{j'}-a_{j'i}/(a_{j'i}-b_{j'j})$  with $i\in I_{0,\alpha}$ are invertible near $(z^0,t^0,u^0)$.
In a neighborhood of $(z^0,t^0,u^0)$, a trivializing section of $K_{\ol{\Ccal}/\Ucal}$ is given by $dz_j/(z_j-b_{jj'})$, while $(z_j-b_{jj'})^{d\mu_{\alpha}}$ is a trivializing section of $-d\mu_{\alpha}F^0_{\alpha}$.
Thus $\Phi_j$ is a trivializing section of $d(K_{\ol{\Ccal}/\Ucal}-\mu_{\alpha}F^0_{\alpha})$  in this neighborhood.

We claim that $t^{\beta_j}$ is a trivializing section of $\sum_{\alpha'\in \NN(\xx)\setminus\{\alpha\}}-d\mu_{\alpha'}F^0_{\alpha'}$.
Since $\beta_{j,\alpha}=0$, we have $t^{\beta_j}=\prod_{\alpha'\in \NN(\xx)\setminus\{\alpha\}} t_{\alpha'}^{\beta_{j,\alpha'}}$.
Then the same argument as the previous case allows us to conclude.
Since $(z^0,t^0,u^0)$ does not belong to any divisor $\Gamma_i$'s, it follows that $t^{\beta_j}\Phi_j$ is a trivializing section of $\Kcal_\mu$ in a neighborhood of $(z^0,t^0,u^0)$.

\medskip

We have then showed that for every $(z^0,t^0,u^0) \in \ol{\Ccal}_{\Ucal}$ there is some $j\in \{0,\dots,r\}$ such that $t^{\beta_j}\Phi_j$ is a trivializing section of $\Kcal_\mu$ in a neighborhood of $(z^0,t^0,u^0)$. Since all the sections $t^{\beta_j}\Phi_j$'s differ from a fixed one, say $t^{\beta_0}\Phi_0$, by a non-vanishing function depending only on $(t,u)$ by Proposition~\ref{prop:Kaw:trivial:ratio}, we conclude that all of them are trivializing sections of $\Kcal_\mu$ on $\ol{\Ccal}_{\Ucal}$.
\end{proof}

\subsection{Proof of Theorem~\ref{th:push:Kaw:ln:bdl}}\label{sec:prf:push:Kaw:ln:bdl}
\begin{proof}
By Proposition~\ref{prop:kaw:trivial:sect}, every point $\xx\in  \ol{\Mod}_{0,n}$ has a neighborhood $\Ucal$ such that the line bundle $\Kcal_\mu$ is trivial over $p^{-1}(\Ucal)$.
Every holomorphic section $\Phi$ of $\Kcal_\mu$ on $p^{-1}(\Ucal)$ is then given by a holomorphic function.
Since all the fibers of $p$ are compact and connected, $\Phi$ must be constant on the fibers.
Thus $\Phi=p^*f$, where $f$ is a holomorphic function on $\Ucal$.
This implies that $p_*\Kcal_{\mu|\Ucal} \sim \Ocal_{\Ucal}$.
Thus $p_*\Kcal_\mu$ is a line bundle $\bar{\Lcal}_\mu$ on $\ol{\Mod}_{0,n}$ and we have $\Kcal_\mu \sim p^*\bar{\Lcal}_\mu$.
It is clear from the definition that the restriction of $\bar{\Lcal}_\mu$ to $\Mod_{0,n}$ is $\Lcal_\mu$.

The expression of the divisor on $\ol{\Mod}_{0,n}$ corresponding to $\bar{\Lcal}_\mu$ can be derived from the Grothendieck-Riemann-Roch formula. Details of the calculation can be found in \cite[\textsection 8.4]{KN18}.
\end{proof}


\section{Proof of Theorem~\ref{th:vol:n:inters:g0}}\label{sec:proof:vol:n:inters:nb}
\begin{proof}
Remark that the equivalence \eqref{eq:except:Weil:div} has been proved in Theorem~\ref{th:coeff:bound:div:in:E}. Let us prove the equivalence \eqref{eq:tauto:ln:bdl:expr}.
Consider a point $\xx$ in a stratum of codimension $r\geq 1$ of $\ol{\Mod}_{0,n}$. Let $\Ucal_\xx, (t,u)$ be as in \textsection\ref{subsec:blowup:construct}.
Recall that $\widehat{\Ucal}_\xx$ is covered by the family of open subsets
$$
\widehat{\Ucal}_\xx^i:=\{(t,u,[v_0:\dots:v_r])\in \widehat{\Ucal}_\xx, \, v_i\neq 0\}, \, i=0,\dots,r.
$$
On each $\widehat{\Ucal}^i_\xx$, the Cartier divisor $\Ecal$ is defined by $t^{\beta_i}$.
In the proof of Theorem~\ref{th:inject:blowup:proj:bdl}, we have seen that a trivialization of  $\OO(-1)_{|\widehat{\Ucal}^i_\xx}$ is given by $\Phi_i$.
Since we can write $\Phi_i=\frac{t^{\beta_i}\Phi_i}{t^{\beta_i}}$, and $t^{\beta_i}\Phi_i$ is a trivialization of the bundle $\bar{\Lcal}_\mu$ over $\Ucal_\xx$ (see Proposition~\ref{prop:kaw:trivial:sect}) and $1/t^{\beta_i}$ can be identified with a trivializing section of the line bundle asociated to $\Ecal$ over $\widehat{\Ucal}^i_\xx$, we get
$$
\hat{\Lc}_\mu:=\OO(-1)_{\Pb\ol{\Hc}^{(d)}_{0,n|\blowupsp}}\sim \hat{p}^*\bar{\Lc}_\mu+\Ec \sim \hat{p}^*D_\mu+\Ec.
$$
Finally, the equality \eqref{eq:vol:n:inters:g0} is an immediate consequence of \cite[Th.1.4 (b)]{Ng22}.
\end{proof}
\begin{Remark}\label{rk:vol:n:inters:CMZ}
Equality \eqref{eq:vol:n:inters:g0} can also be derived from the results of \cite{CMZ19}.
\end{Remark}

\end{document}